\documentclass[12pt]{article}

\usepackage{amsmath}
\usepackage{graphicx,psfrag,epsf}
\usepackage{enumerate}
\usepackage{natbib}
\usepackage{url} 

\usepackage{amssymb, amsfonts, amscd, xspace, pifont, amsthm}
\usepackage{mathrsfs}
\usepackage{booktabs}
\usepackage{multirow}
\usepackage{epstopdf}
\usepackage{color}
\usepackage[lined,boxed,commentsnumbered]{algorithm2e}
\usepackage{tikz}
\usepackage{enumitem}   
\usetikzlibrary{fit,positioning,arrows,automata}
\usepackage{setspace}

\newcommand{\V}[1]{\ensuremath{\boldsymbol{#1}}\xspace}

\newtheorem{thm}{Theorem}
\newtheorem{lemma}{Lemma}
\newtheorem{prop}{Proposition}

\def\twoImages#1#2#3#4#5#6 
{
	\centerline{\hfill\makebox[#2]{#3}\hfill\makebox[#5]{#6}\hfill}
	\centerline{\hfill
		\includegraphics[width=#2]{#1}
		\hfill
		\includegraphics[width=#5]{#4}
		\hfill}
}

\newenvironment{definition}[1][Definition]{\begin{trivlist}
		\item[\hskip \labelsep {\bfseries #1}]}{\end{trivlist}}

\DeclareMathOperator*{\argmax}{arg\,max}

\newtheorem{manualtheoreminner}{Theorem}
\newenvironment{manualtheorem}[1]{%
	\renewcommand\themanualtheoreminner{#1}%
	\manualtheoreminner
}{\endmanualtheoreminner}

\begin{document}
\title{Identifiability and Consistency of Network Inference Using the Hub Model and Variants}
\date{}
\author{
	Yunpeng Zhao  \\ 
    Arizona State University \\
	\and
	Peter Bickel \\
	University of California, Berkeley \\
	\and
	Charles Weko\\
	U.S. Army \\
}
\maketitle

\begin{abstract}
	Statistical network analysis primarily focuses on  inferring  the parameters of an observed network. In many applications, especially in the social sciences, the observed data is the groups  formed by individual subjects.   In these applications, the network is itself a parameter of  a statistical model. Zhao and Weko (2019) propose a model-based approach, called the \textit{hub model}, to infer implicit networks from grouping behavior. The  hub model  assumes  that  each member of the  group is brought together by a member of the group called the \textit{hub}. The hub model belongs to the family of Bernoulli mixture models.  Identifiability of parameters is a notoriously difficult problem for Bernoulli mixture models. This paper proves  identifiability of the hub model parameters  and estimation consistency under mild conditions. Furthermore,  this paper generalizes the hub model by  introducing a  model component  that allows  hubless groups in which individual nodes  spontaneously appear  independent of any other individual.  We refer to this additional component as the \textit{null component}.  The new model bridges the gap between the hub model and the degenerate case of the mixture model -- the Bernoulli product.  Identifiability and consistency are  also proved for the new model. Numerical studies are provided to demonstrate the theoretical results.
\end{abstract}


\section{INTRODUCTION}

In recent years, network analysis has been applied in  science and engineering fields including mathematics, physics, biology, computer science, social sciences and statistics (see \cite{Getoor2005, Goldenberg2010, Newman2010} for reviews). Traditionally, statistical network analysis deals with parameter estimation of an observed network, i.e., an observed adjacency matrix. For example, community detection, a topic of broad interest, studies how to partition the node set of an observed network into cohesive overlapping or non-overlapping communities (see \cite{abbe2017community,zhao2017survey} for recent reviews). Other well-studied statistical network models include the preferential attachment model \citep{Barabasi&Albert1999}, exponential random graph models \citep{Frank&Strauss1986,robins2007introduction}, latent space models \citep{Hoff2002,Hoff2007}, and the graphon model \citep{diaconis2007graph,gao2015rate,zhang2017estimating}.

In contrast to traditional statistical network analysis, this paper focuses on  inferring a latent network structure. Specifically, we model data with the following format: each observation in the dataset is a subset of nodes that are observed simultaneously.  An observation is called a \textit{group} and  a full dataset  is called \textit{grouped data}.  \cite{Wasserman94} introduced this format using the toy example of a children's birthday party. In  their simple example, children are treated as nodes and  each party represents a group -- i.e., a subset of children who attended the same party is a group. The reader is referred to \cite{Zhao2015,weko2017penalized} for applications of such data to the social sciences and animal behavior. 

The observed grouping behavior presumably results from a latent social structure that can be interpreted as a  network structure of associated individuals \citep{Moreno34}. The  task is therefore to infer  a latent network structure from grouped data. Existing methods  mainly focus on ad-hoc descriptive approaches from the social sciences literature, such as the co-occurrence matrix \citep{Wasserman94} or the half weight index \citep{Cairns87}.  \cite{Zhao2015} propose the first model-based approach, called the \textit{hub model}, which assumes that every observed group has a \textit{hub} that brings together the other members of the group. 

 \cite{Zhao2015} demonstrated the hub model by analyzing co-sponsorship of legislation in the Senate of the 110$^{th}$ United States Congress.  The rules of the Senate require that each piece of legislation have a single, unique sponsor; however, other members may co-sponsor the bill.  These rules mean that the legislation sponsorship data conforms to the hub model assumption that every group has a single  hub. Analyzing this data is trivial when the sponsors are known, i.e., when the hubs are observed; \cite{Zhao2015} also estimated the latent network when sponsorship data was eliminated from the data. In general, when the hub nodes of grouped data are known, estimating the model parameters is a trivial task.  In most research situations, hub nodes are unknown and need to be modeled as latent variables. Under this setup, estimating the model parameters becomes a more difficult task.

This paper has two aims: first, to prove the identifiability of the canonical parameters for the hub model, i.e., the probabilities of being a hub node of a group and the probabilities of being included in a group  formed by a particular hub node, and the asymptotic consistency for the estimators. Both results are proved when hubs are unobserved. The second aim is to expand the hub model  definition to allow for hubless groups. Identifiability and consistency are also proved for this new model.  

Given a parameterization, the parameters of a family of distributions are  identifiable if different values of the parameter must correspond to different distributions. The hub model is a restricted class from the family of finite mixtures of multivariate Bernoulli \citep{Zhao2015}.  \cite{gyllenberg1994non} showed that in general the parameters of finite mixture models of multivariate Bernoulli are  not identifiable.  \cite{Zhao2015} showed that the canonical parameters of the hub model are identifiable under two assumptions.  Their first assumption is that the hub node of each group always appears in the group it  forms.  They further gave a counterexample that shows the parameters are not identifiable if only this first assumption is imposed.  To achieve identifiability, they applied a second assumption requiring that relationships be reciprocal   (that is, the adjacency matrix is symmetric). 

In applications where prediction accuracy is the main goal \citep{carreira2000practical}, identifiability is not necessary because the exact values of the estimated parameters are not used in follow-on analysis.  However, to meet the needs of researchers in the social sciences and animal behavior, identifiability is a prerequisite.  That is, independent researchers analyzing the same data must obtain the same estimates.  Additionally,  for statistical inference  to yield parameters that can  be consistently estimated, identifiability is a necessary condition \citep{allman2009identifiability}.  

This paper considers identifiability of canonical parameters when adjacency matrices are not symmetric. The model is therefore referred as to the \textit{asymmetric hub model}.  We prove that when the hub set (i.e., the set of possible hubs) contains at least one fewer member than the node set, the parameters are identifiable under simple conditions. We argue that  this new setup is practical and less restrictive than earlier model conditions. As pointed out by \cite{weko2017penalized}, it is not necessary to  include every node in the population  in the hub set because there may be low ranking members who do not have  the influence to initiate a group. Moreover, allowing the hub set to be smaller than the node set can reduce model complexity. 

In addition to identifiability, we prove the consistency of estimators for the hub model.      As mentioned before, given an estimate of the hub nodes, estimating the model parameters is trivial.  Therefore, in our consistency proofs, we first prove the consistency of the hub estimates and then show that the estimators of model parameters are consistent as a corollary. We consider the most general setup in which the number of groups (i.e., sample size), the size of the node set, and the size of the hub set are all allowed to grow.  Since the hub model is a mixture model, estimation of the latent hub for each group can be viewed as a clustering problem.  That is, the latent hub of a group can be viewed as the class label of the group. We therefore borrow the technique of profile likelihoods from the community detection literature \citep{Bickel&Chen2009,Choietal2011} and prove the consistency for hub estimation. The consistency of parameter estimation then holds as a corollary. 

As mentioned above, the second aim  of this paper is to generalize the hub model  to accommodate  hubless groups and then prove identifiability and consistency of this generalized model. The classical hub model requires each group to have a  hub. As observed in \cite{weko2017penalized}, when fitting the hub model to data, one sometimes has to choose an unnecessarily large hub set due to this requirement. For example, a node that appears infrequently in general but appears once  as a singleton in a group must be included in the hub set simply because  it  must be the hub node of that group. To relax the \textit{one-hub} restriction, we add a  component to the hub model that allows  hubless groups in which nodes appear independently. We call this additional component the \textit{null component} and call the new model the \textit{hub model with a null component}. The null component creates a natural connection between the hub model and a null model. That is, if the hub set is empty then the model degenerates to the model in which nodes appear independently in groups. That is, each group is generated by independent Bernoulli trials. The proofs of identifiability and consistency for the hub model with a null component do not parallel the first set of proofs and are more challenging.

For a brief review of other related work, we recommend the following literature. The study of identifiability under finite mixture models dates back to the 1960s \citep{teicher1963identifiability,yakowitz1968identifiability}.  \cite{gyllenberg1994non} showed that finite mixtures of Bernoulli products are not identifiable.  \cite{allman2009identifiability} introduced and studied the concept of generic identifiability, which means that the set of non-identifiable parameters has Lebesgue measure zero. By contrast, we study identifiability defined in the strictest sense. Identifiability under another class of mixture Bernoulli models has been recently studied \citep{xu2017identifiability,gu2019sufficient}. This class of models, for example, the DINA  (Deterministic Input, Noisy ``And'' gate) model, has applications in psychological and educational research. The motivation, the model setup, and the proof techniques presented here  are all different from previous research, and the result of neither implies the other. 

The remainder of this article is organized as follows. The structures of Section \ref{sec:classical} and Section \ref{sec:null} are parallel. In each section, we first present the model setup and then give the identifiability and consistency results. The behavior of the estimators is evaluated by numerical studies in Section \ref{sec:numerical}. Section \ref{sec:summary} concludes with a summary and discussion.  Technical proofs are given in the Appendix. 

\section{THE ASYMMETRICAL HUB MODEL} \label{sec:classical}
\subsection{Model setup} \label{sec:classical_setup}

First, we review the grouped data structure and propose a modified version of the hub model, called the \textit{asymmetric hub model}. For a set of $n$ individuals, $V=\{1,..., n\}$, we observe $T$ subsets, called \textit{groups}.

In this paper,  groups are treated as a random sample of size $T$ with each group being an  observation. Each  group  is represented by an $n$ length row vector $G^{(t)}$, where
\[ G_i^{(t)} = \left\{ 
\begin{array}{l l}
1 & \quad \textnormal{if node $i$ appears in group $t$,}\\
0 & \quad \textnormal{otherwise, }
\end{array} \right.\]
for $i=1,...,n$ and $t=1,...,T$. 
The full dataset is a $T \times n$ matrix $\V{G}$ with $G^{(t)}$ being its rows. 

Let $V_L$ be the set of all nodes which can serve as a hub and let $n_L=|V_L|$.  We refer to $V_L$ as the \textit{hub set} and call the nodes in this set \textit{leaders}.

In contrast to  the setup in \cite{Zhao2015} where the  hub set contains all nodes, we assume  that the hub set may contain fewer members than  the whole set of nodes. In other words, the hub set $V_L$ may be a subset of $V$.  For  simplicity of notation, we further assume $V_L=\{1,...,{n_L}\}$. We refer to nodes from $n_L+1$ to $n$ as \textit{followers}. 

Given this notation,  the true hub of $G^{(t)}$ is represented by $z_{*}^{(t)}$ and takes on values from $1,...,n_L$.

Under the hub model, each group $G^{(t)}$ is independently generated by the following two-step process:
\begin{enumerate}[label=(\roman*)]
	\item The hub is sampled from a multinomial trial with parameter $\rho=(\rho_1,...,\rho_{n_L})$, i.e., $\mathbb{P}(z_*^{(t)}=i)=\rho_i$, with $\sum_{i=1}^{n_L} \rho_i=1$.
	
	\item Given  the hub node $i$,  each node $j$ appears in the group independently with probability $A_{ij}$, i.e., $	\mathbb{P}(G_j^{(t)}=1|z_*^{(t)}=i)=A_{ij}$.  
\end{enumerate}

Before proceeding, there are a number of implications of the proceeding terms and notation.  We interpret $\rho_i$ to be the probability that node $i$ is the hub of a group and $A_{ij}$ to be the probability that node $j$ is a member of a group given node $i$ is the hub of the group.  Thus, the term \textit{leader} applies to any node $i$ with a non-zero $\rho_i$ and the term \textit{follower} applies to any node $j$ with $\rho_j=0$.  Additionally, observe that multiple leaders may appear in the same group although only one of them will be the hub of that group.  That is, for two nodes $i$ and $j$ where $\rho_i>0$ and $\rho_j>0$, $A_{ij}$ may be non-zero.


A key assumption  from \cite{Zhao2015} which we adopt in this paper is that a hub node must appear in any group that it forms (i.e., $A_{ii}\equiv 1$, for $i=1,...,n_L$). The parameters for the hub model are thus
\begin{align*}
\rho & =(\rho_1,...,\rho_{n_L}), \\
A_{n_L\times n} & = \begin{pmatrix}
1 &  A_{12} & \cdots & A_{1, n_L}  & A_{1,n_L+1} & \cdots & A_{1,n}\\
A_{21} & 1 & \cdots & A_{2,n_L} & A_{2,n_L+1} & \cdots & A_{2,n} \\
\vdots & \vdots & \ddots & \vdots & \vdots & \ddots & \vdots  \\ 
A_{n_L,1} & A_{n_L,2} & \dots & 1 & A_{n_L,n_L+1} & \cdots & A_{n_L,n}
\end{pmatrix}.
\end{align*}
As in \cite{Zhao2015}, we interpret $A_{ij}$ as the strength of the relationship between node $i$ and $j$. We differ from \cite{Zhao2015} in that   $A$ may be a non-square matrix and $A_{ij}$ is not necessarily equal to $A_{ji}$. The setting in this article is more natural.   Social relationships are usually non-reciprocal and in most organizations there  are members who do not have the authority or  willingness to initiate  groups. Later in this paper, we will show how this new setup presents challenges for theoretical analysis but is feasible.

We begin with the case where both $G$ and $z_*$ are observed. The likelihood function  is
\begin{equation*}
\mathbb{P}(\V{G},z_*|A,\rho)=\prod_{t=1}^T \prod_{i=1}^{n_L} \prod_{j=1}^n \big[A_{ij}^{G_j^{(t)}} (1-A_{ij})^{(1-G_j^{(t)})}\big]^{1(z_*^{(t)}=i)}\prod_{i=1}^{n_L}\rho_i^{1(z_*^{(t)}=i)},
\end{equation*}
where $1(\cdot)$ is the indicator function. With both $G$ and $z_*$ being observed, it is straightforward to estimate $A$ and $\rho$ by the maximum likelihood estimator:
\begin{align*}
\hat{A}_{ij}^{z_*}= & \frac{\sum_t G_j^{(t)} 1(z_*^{(t)}=i) }{\sum_t 1(z_*^{(t)}=i)}, \quad  i=1,...,n_L,j=1,...,n, \\
\hat{\rho}_i^{z_*}= & \frac{\sum_t 1(z_*^{(t)}=i)}{T}, \quad i=1,...,n_L.
\end{align*}

When the hub node of each group is latent, i.e., when $z_*$ is unobserved, the estimation problem becomes challenging and is the focus in this paper. Integrating out $z_{*}$, the marginal likelihood of $\V{G}$ is
\begin{equation}\label{likelihood_classical}
\mathbb{P}(\V{G}|A,\rho)=\prod_{t=1}^T \sum_{i=1}^{n_L} \rho_i \prod_{j=1}^n {A_{ij}^{G_j^{(t)}} (1-A_{ij})^{1-G_j^{(t)}}},
\end{equation}
which has the form of a Bernoulli mixture model. Hereafter the term hub model refers to the case where $z_*$ is unobserved, unless otherwise specified. 

As will be seen in Section \ref{sec:classical_consistency}, a key feature of parameter estimation when the hub nodes are unknown is estimating the hub node for each group. 

Before considering estimation  of $\rho$ and $A$, we need to  establish the identifiability of parameters $\rho$ and $A$ under \eqref{likelihood_classical}. If there exist two different sets of parameters that can give the same likelihood then the parameters are not estimable.  \cite{Zhao2015} observed that when every node is allowed to be a  leader (i.e., $n_L=n$), $\rho$ and $A$ are not identifiable without additional restrictions. They further proved that the symmetry of $A$ is a sufficient condition for identifiability. We remove this stringent constraint and seek a set of milder identifiability conditions in the next section. 
\subsection{Identifiability under the  hub model}
To precisely define identifiability, let $\mathcal{P}$ be the parameter space where $\mathcal{P}=\{(\rho,A) | 0\leq \rho_i \leq 1; A_{ii} = 1; 0\leq A_{ij}\leq 1, i=1,...,n_L, j=1,...,n, i\neq j   \}$. Let $\V{g}=(g_i^{(t)})$ be any realization of $\V{G}$ under the hub model.
\begin{definition}\label{def:identi}
	The parameters $(\rho,A)$ are identifiable under the hub model if the following holds:
	\begin{align*}
	\forall g, \forall (\tilde{\rho},\tilde{A}) \in \mathbb{P}(\V{G}=\V{g}|\rho,A)= \mathbb{P}(\V{G}=\V{g}|\tilde{\rho},\tilde{A}) \iff  (\rho,A) = (\tilde{\rho},\tilde{A}).
	\end{align*} 
\end{definition} 

Note that we define identifiability in the strictest sense  and the above definition does not allow label swapping of latent classes.  In cluster analysis label swapping refers to the fact that nodes can be successfully partitioned into latent classes, but individual classes cannot be uniquely identified.  For example, community detection may correctly partition voters into communities based on their political preferences, but cannot identify which political party each community prefers.  This is not an issue in the hub model due to the constraint $A_{ii}=1$. In addition, note that we only need to consider identifiability for the distribution of a single observation, i.e., $T=1$ because the data are independently and identically distributed. Let $g$ be a realization of a single observation hereafter.

We now give the identifiability result for the asymmetric hub model. 

\begin{thm}\label{thm:hub_iden}
	The parameters $(\rho,A)$ of the hub model are identifiable under the following conditions:
	\begin{enumerate}[label=(\roman*)]
		\item $0<\rho_i<1$, for $i=1,...,n_L$;
		\item $A_{ij}<1$, for $i=1,...,n_L, j=1,...,n, i\neq j$;
		\item for all $i=1,...,n_L$, $i'=1,...,n_L,i\neq i'$, there exists $k \in \{n_L+1,...,n\}$ such that $A_{ik} \neq A_{i'k}$.
	\end{enumerate}
\end{thm}

Before proving Theorem \ref{thm:hub_iden}, we make the following remarks. First, note that the conditions only apply to  the true parameters $\rho$ and $A$ but not to $\tilde{\rho}$ and $\tilde{A}$, except that $(\tilde{\rho},\tilde{A}) \in \mathcal{P}$. Second, condition (\romannumeral 3) implies that there exists at least one follower in the node set, i.e., $n_L<n$, and for any pair of  nodes in the hub set, there exists a follower with different probability of being included in groups formed by the two hubs, respectively. 
\begin{proof}[Proof of Theorem \ref{thm:hub_iden}]
	Let $(\tilde{\rho},\tilde{A})\in \mathcal{P}$ be a set of parameters such that $\mathbb{P}(g|\rho,A)= \mathbb{P}(g|\tilde{\rho},\tilde{A})$ for all $g$. 	For all $i=1,...,n_L$, $k=n_L+1,...,n$, consider the probability that only $i$ appears under parameterizations $(\rho,A)$ and $(\tilde{\rho},\tilde{A})$, respectively
	\begin{align*}
	\tilde{\rho}_i (1-\tilde{A}_{ik}) \prod_{j=1,...,n, j \neq i,j \neq k} (1-\tilde{A}_{ij}) & = \rho_i (1-A_{ik}) \prod_{j=1,...,n, j \neq i,j \neq k} (1-A_{ij}),
	\end{align*}
	and the probability that only $i$ and $k$ appear
	\begin{align*}
	\tilde{\rho}_i \tilde{A}_{ik}  \prod_{j=1,...,n,j \neq i, j \neq k} (1-\tilde{A}_{ij})  & =\rho_i A_{ik} \prod_{j=1,...,n,j \neq i, j \neq k} (1-A_{ij}).
	\end{align*}
	Dividing the second equation by the first, we obtain $\tilde{A}_{ik}/(1-\tilde{A}_{ik})=A_{ik}/(1-A_{ik})$ and hence $\tilde{A}_{ik}=A_{ik}$ for $i=1,...,n_L$, $k=n_L+1,...,n$.

	For any $i=1,...,n_L$, $i'=1,...,n_L$, $i\neq i'$, suppose that $k$ is the follower such that $A_{ik} \neq A_{i'k}$. Consider the probability that only $i$ and $i'$ appear  
	\begin{align*}
	&\tilde{\rho}_i \tilde{A}_{ii'} (1-\tilde{A}_{ik}) \prod_{j=1,...,n,j \neq i, j \neq i',j\neq k}(1-\tilde{A}_{ij})+\tilde{\rho}_{i'} \tilde{A}_{i'i}(1-\tilde{A}_{i'k})\prod_{j=1,...,n,j \neq i, j \neq i',j\neq k} (1-\tilde{A}_{i'j})  \\
	= & \rho_i A_{ii'} (1-A_{ik}) \prod_{j=1,...,n,j \neq i, j \neq i',j\neq k}(1-A_{ij})+\rho_{i'} A_{i'i}(1-A_{i'k})\prod_{j=1,...,n,j \neq i, j \neq i',j\neq k} (1-A_{i'j}),  
	\end{align*}
	and the probability that $i$, $i'$ and $k$ appear
	\begin{align*}
	&\tilde{\rho}_i \tilde{A}_{ii'} \tilde{A}_{ik}\prod_{j=1,...,n,j \neq i, j \neq i',j\neq k}(1-\tilde{A}_{ij})+\tilde{\rho}_{i'} \tilde{A}_{i'i}\tilde{A}_{i'k}\prod_{j=1,...,n,j \neq i, j \neq i',j\neq k} (1-\tilde{A}_{i'j})  \\
	= & \rho_i A_{ii'} A_{ik} \prod_{j=1,...,n,j \neq i, j \neq i',j\neq k}(1-A_{ij})+\rho_{i'} A_{i'i}A_{i'k}\prod_{j=1,...,n,j \neq i, j \neq i',j\neq k} (1-A_{i'j}).
	\end{align*}
	As $\tilde{A}_{ik}=A_{ik}$ for $i=1,...,n_L$, $k=n_L+1,...,n$, the above two equations become
	\begin{align}
	&\tilde{\rho}_i \tilde{A}_{ii'} (1-A_{ik}) \prod_{j=1,...,n,j \neq i, j \neq i',j\neq k}(1-\tilde{A}_{ij})+\tilde{\rho}_{i'} \tilde{A}_{i'i}(1-A_{i'k})\prod_{j=1,...,n,j \neq i, j \neq i',j\neq k} (1-\tilde{A}_{i'j}) \label{t1} \\
	= & \rho_i A_{ii'} (1-A_{ik}) \prod_{j=1,...,n,j \neq i, j \neq i',j\neq k}(1-A_{ij})+\rho_{i'} A_{i'i}(1-A_{i'k})\prod_{j=1,...,n,j \neq i, j \neq i',j\neq k} (1-A_{i'j}), \nonumber  \\
	&\tilde{\rho}_i \tilde{A}_{ii'} A_{ik} \prod_{j=1,...,n,j \neq i, j \neq i',j\neq k}(1-\tilde{A}_{ij})+\tilde{\rho}_{i'} \tilde{A}_{i'i}A_{i'k}\prod_{j=1,...,n,j \neq i, j \neq i',j\neq k} (1-\tilde{A}_{i'j}) \label{t2} \\
	= & \rho_i A_{ii'} A_{ik} \prod_{j=1,...,n,j \neq i, j \neq i',j\neq k}(1-A_{ij})+\rho_{i'} A_{i'i}A_{i'k}\prod_{j=1,...,n,j \neq i, j \neq i',j\neq k} (1-A_{i'j}). \nonumber
	\end{align}
	Eq. \eqref{t1} and \eqref{t2} can be viewed as a system of linear equations with unknown variables
	\begin{align*}
	\tilde{\rho}_i \tilde{A}_{ii'}  \prod_{j=1,...,n,j \neq i, j \neq i',j\neq k}(1-\tilde{A}_{ij})
	\end{align*}
	and 
	\begin{align*}
	\tilde{\rho}_{i'} \tilde{A}_{i'i}\prod_{j=1,...,n,j \neq i, j \neq i',j\neq k} (1-\tilde{A}_{i'j}).
	\end{align*} 
	As $A_{ik} \neq A_{i'k}$, the system has full rank and hence has one and only one solution: 
	\begin{align} 
	\tilde{\rho}_i \tilde{A}_{ii'}  \prod_{j=1,...,n,j \neq i, j \neq i',j\neq k}(1-\tilde{A}_{ij}) & =\rho_i A_{ii'}  \prod_{j=1,...,n,j \neq i, j \neq i',j\neq k}(1-A_{ij}), \label{hub_final} \\
	\tilde{\rho}_{i'} \tilde{A}_{i'i}\prod_{j=1,...,n,j \neq i, j \neq i',j\neq k} (1-\tilde{A}_{ij}) & =\rho_{i'} A_{i'i}\prod_{j=1,...,n,j \neq i, j \neq i',j\neq k} (1-A_{i'j}). \nonumber
	\end{align}
	
	Combining \eqref{hub_final} with 
	\begin{align*}
	\tilde{\rho}_i (1-\tilde{A}_{ii'})  \prod_{j=1,...,n,j \neq i, j \neq i',j\neq k}(1-\tilde{A}_{ij}) & =\rho_i (1-A_{ii'})  \prod_{j=1,...,n,j \neq i, j \neq i',j\neq k}(1-A_{ij})
	\end{align*}
	we obtain $\tilde{A}_{ii'}=A_{ii'}$ for $i=1,...,n_L, i'=1,...,n_L$ by a similar argument to that at the beginning of the proof. It follows immediately that $\tilde{\rho}_i=\rho_i$ for $i=1,...,n_L$.
\end{proof}

\subsection{Consistency of the maximum profile likelihood estimator for the hub model}\label{sec:classical_consistency}

We consider the asymptotic consistency for the hub model in the most general setting. That is, we allow the number of groups ($T$), the size of the node set ($n$), and the size of the hub set ($n_L$) to grow. As mentioned before, we  reformulate the problem as a clustering problem where a cluster is defined as the groups  formed by the same hub node.  We borrow the techniques from the community detection literature to prove the consistency of class labels, i.e., the consistency of hub labels. The consistency of parameter estimation then holds as a corollary. 

Let $z=(z^{(t)})_{t=1,...,T}$ be an arbitrary assignment of hub labels. Given $z$, the log-likelihood of the full dataset $\V{G}$ is
\begin{align}\label{cond_lik}
L_G(A|z)=\sum_{t=1}^T \sum_{j=1}^n G_j^{(t)} \log A_{z^{(t)},j}+ (1- G_j^{(t)} ) \log (1-A_{z^{(t)},j} ).
\end{align}
For $i=1,...,n_L$, let $t_i=\sum_t 1(z^{(t)}=i)$  be the number of groups with hub $i$. 
Given $z$, the maximum likelihood estimator of $A$ is  
\begin{align*}
\hat{A}_{ij}^z=\frac{\sum_t G_j^{(t)} 1(z^{(t)}=i) }{t_i }.
\end{align*}
We will omit the upper index $z$ when it is clear from the context.   Plugging $\hat{A}_{ij}$ back into \eqref{cond_lik}, we obtain the profile log-likelihood 
\begin{align*}
L_G(z)=\max_{A} L_G(A|z)=\sum_t \sum_j G_j^{(t)} \log \hat{A}_{z^{(t)},j}+ (1- G_j^{(t)} ) \log (1-\hat{A}_{z^{(t)},j} ).
\end{align*}

Furthermore, let 
\begin{align*}
\hat{z}=\argmax_{z} L_G(z).
\end{align*}
The framework of profile likelihoods are adopted in the community detection literature \citep{Bickel&Chen2009,Choietal2011}, where $z$ is treated as an unknown parameter and we search for the $z$  that optimizes the profile likelihood. 

Recall that $z_*$ is the true class assignment.    We will treat $z_*$ as a random vector.  Although this treatment  makes the proof slightly more complicated, it maintains continuity with the previous section. The same consistency result holds when $z_*$ is treated as fixed.   Let $t_{i*}=\sum_t 1(z_*^{(t)}=i)$, for $i=1,...,n_L$. 

Let $P_j^{(t)}=P(G^{(t)}_j=1|z_{*}^{(t)})=A_{z_*^{(t)},j}$. Then by replacing $G_j^{(t)}$ by $P_j^{(t)}$, we obtain a ``population version'' of $L_G(z)$: 
\begin{align*}
L_P(z)=\sum_t \sum_j P_j^{(t)} \log \bar{A}_{z^{(t)},j}+ (1- P_j^{(t)} ) \log (1-\bar{A}_{z^{(t)},j} ),
\end{align*}
where 
\begin{align}\label{def_A_bar}
\bar{A}_{ij}=\frac{\sum_t P_j^{(t)} 1(z^{(t)}=i) }{t_i }. 
\end{align}

Let $T_e=\sum_t 1(z_*^{(t)} \neq \hat{z}^{(t)})$ be the number of groups with  incorrect hub labels. As discussed previously, we do not allow label swapping in the definition of $T_e$.  Our aim is to prove the following result: 
\begin{align*}
T_e/T=o_p(1), \quad \mbox{as } n_L\rightarrow \infty, n\rightarrow \infty, T\rightarrow \infty.
\end{align*}

Note that the statement above is well-defined even though we assume $z_*$ is a random vector in this treatment.  We prove the consistency of $\hat{z}$ by a typical proof technique for consistency of M-estimators. That is, the consistency of $\hat{z}$ holds by proving a uniform bound for $|L_G(z)-L_P(z)|$ and proving that $T_e/T$ can be bounded by $L_P(z_*)-L_P(\hat{z})$. 

We begin with a lemma that gives bounds for $t_{i*}$.  
\begin{lemma}\label{thm:rho}
	Assume $\min\{\rho_1,...,\rho_{n_L}\}=c_{\textnormal{min}*} /n_L$ and $\max\{\rho_1,...,\rho_{n_L}\}=c_{\textnormal{max}*} /n_L$ where $c_{\textnormal{min}*}$ and $c_{\textnormal{max}*}$ are positive constants. When $n_L^2 (\log n_L)/T=o(1)$, there exist positive constants $c_\textnormal{min}$ and $c_{\textnormal{max}}$ such that 
	\begin{align*}
	\frac{Tc_{\textnormal{min}}}{n_L} \leq t_{i*} \leq \frac{Tc_{\textnormal{max}}}{n_L}, \quad i=1,...,n_L,
	\end{align*} 
	with probability approaching 1. 
\end{lemma}

The proofs of the theorems in this section are given in the Appendix. Lemma \ref{thm:rho} states that the number of groups formed by a given hub node grows at a rate proportional to $T/n_L$ and all asymptotic results below are under this condition.  

The next lemma decomposes $L_G(z)-L_P(z)$ into two terms.

\begin{lemma}\label{thm:partition}
	Let $Ber(\cdot)$ denote the Bernoulli distribution and let $D(\hat{A}_{ij}|\bar{A}_{ij})$ be the Kullback--Leibler divergence of $Ber(\hat{A}_{ij})$ and $Ber(\bar{A}_{ij})$.  Then for all $z$,
	\begin{align}
	L_G(z)-L_P(z)=\sum_{i=1}^{n_L} t_i \sum_{j } D(\hat{A}_{ij}|\bar{A}_{ij})+B_{n_L,n,T}  \label{partition_weko}
	\end{align}
	where 
	\begin{align*}
	B_{n_L,n,T}=\sum_{i=1}^{n_L} t_i \left ( \sum_{j } (\hat{A}_{ij}-\bar{A}_{ij}) \log \frac{\bar{A}_{ij}}{1-\bar{A}_{ij}} \right ).
	\end{align*}
\end{lemma}

\begin{thm}\label{thm:uniform}
	For all $\eta>0 $,
	\begin{align*}
	\mathbb{P}(\max_z |L_G(z)-L_P(z)|\geq 2 \eta ) \leq n_L^T (T/n_L+1)^{n_L n} e^{-\eta}+4 n_L^T \exp \left \{ -\frac{\eta^2}{4(2nT+\eta)} \right \}.
	\end{align*}
\end{thm}
The bound of the first term in \eqref{partition_weko} is proved by a similar argument as in \cite{Choietal2011}. The key challenge is to bound the second term, $B_{n_L,n,T}$. The classical Hoeffding's inequality or Bernstein's inequality cannot be applied because $\log \frac{\bar{A}_{ij}}{1-\bar{A}_{ij}} $ may not be bounded uniformly on $z$ when  $P_j^{(t)}=1$ for many $j$.  \cite{zhao_bernstein} showed that the boundness\footnote{The boundness of $\log \frac{\bar{A}_{ij}}{1-\bar{A}_{ij}} $ for the asymmetric hub model can in fact be proved under certain technical conditions. However, the result does not hold for the hub model with the null component. Therefore, we simply apply the inequality in \cite{zhao_bernstein} to both cases. } of $\log \frac{\bar{A}_{ij}}{1-\bar{A}_{ij}} $ is a technical matter that can be circumvented, and proved a new Bernstein-type inequality that can be applied to our case. 

Now we state the result that $T_e/T$ is bounded by $L_P(z_*)-L_P(\hat{z})$. That is, $z_*$ is a \textit{well-separated} point of maximum of  $L_P$. The reader is referred to Section 5.2 in \cite{van2000asymptotic} for the classical case of this concept. 
\begin{thm}\label{thm:separate}
	Assume
	\begin{enumerate}[label=(\roman*)]
		\item there exists a set $V_i \subset \{1,...,n\}$ for $i=1,...,n_L$ such that $|V_i| \geq v n/n_L$ and $A_{ij}-A_{i'j} \geq d$ for all $j \in V_i$, $i\neq i'$, where $|\cdot|$ is the cardinality  of a set;
		\item  $A_{ii'}\leq c_0/n_L $ for $i=1,...,n_L$, $i'=1,...,n_L$, $i\neq i'$, where $c_0$ is a positive constant.
	\end{enumerate}
	Then if $(n_L^2 \log n_L)/T=o(1)$, for some positive constant $\delta$, 
	\begin{align*}
	\mathbb{P} \left (\frac{\delta  n_L }{d^2 v nT}(L_P(z_*)-L_P(\hat{z})) \geq \frac{T_e}{T} \right )\rightarrow 1, \quad \mbox{as } n_L\rightarrow \infty, n\rightarrow \infty, T\rightarrow \infty.
	\end{align*}
\end{thm}

Condition (\romannumeral 1) implies that for every  leader there exists a set of nodes  that are more likely  to join  groups initiated by this particular leader than any other leader.    The size of this set is influenced by $v$ and the magnitude of this preference is influenced by $d$.  Here $d$ and $v$ can be either fixed numbers between 0 and 1 or quantities that go to 0. The rates will be specified when we give  results of label and parameter estimation consistency later in this section. 

Condition (\romannumeral 2) is a technical condition  that  prevents label swapping from influencing the following consistency results. 

Below we provide two versions of a theorem when applicable -- one  allows $n_L$  to go to infinity and the other assumes $n_L$ is fixed. We omit the proof for fixed $n_L$  because the proof is a trivial corollary of the case when $n_L$ goes to infinity.

\begin{manualtheorem}{\ref{thm:separate}$'$}\label{thm:separate_fixed}
	Assume
	\begin{enumerate}[label=(\roman*)]
		\item there exists a set $V_i \subset \{1,...,n\}$ for $i=1,...,n_L$ such that $|V_i| \geq v n/n_L$ and $A_{ij}-A_{i'j} \geq d$ for all $j \in V_i$, $i\neq i'$;
		\item  $A_{ii'}$ is bounded away from 1 for $i=1,...,n_L$, $i'=1,...,n_L$, $i\neq i'$.
	\end{enumerate}
	Then for some positive constant $\delta$, 
	\begin{align*}
	\mathbb{P} \left (\frac{\delta  }{d^2 v nT}(L_P(z_*)-L_P(\hat{z})) \geq \frac{T_e}{T} \right )\rightarrow 1, \quad \mbox{as }  n\rightarrow \infty, T\rightarrow \infty.
	\end{align*}
\end{manualtheorem}

Combining Theorem \ref{thm:uniform} and Theorem \ref{thm:separate} (Theorem \ref{thm:separate_fixed}), we establish label consistency: 
\begin{thm}\label{thm:label_consistency}
	Under the conditions of Theorem \ref{thm:separate}, if $(n_L^2 \log n_L)/T=o(1)$, $ (n_L^2\log T)/(d^2 v T)=o(1) $ and $ (n_L^2 \log n_L )/ (d^4 v^2 n)=o(1) $,
	\begin{align*}
	T_e/T=o_p(1), \quad \mbox{as } n_L\rightarrow \infty, n\rightarrow \infty, T\rightarrow \infty.
	\end{align*}
\end{thm}
\begin{manualtheorem}{\ref{thm:label_consistency}$'$}
	Under the conditions of Theorem \ref{thm:separate_fixed}, if $(\log T)/(d^2 v T)=o(1) $ and $d^4 v^2 n\rightarrow \infty $,
	\begin{align*}
	T_e/T=o_p(1), \quad \mbox{as }  n\rightarrow \infty, T\rightarrow \infty.
	\end{align*}
\end{manualtheorem}

If we further assume $d$ and $v$ to be constants, we can give the cleanest version of the label consistency result:
\begin{manualtheorem}{\ref{thm:label_consistency}$''$}\label{thm:clean}
	Under the conditions of Theorem \ref{thm:separate_fixed} with $d$ and $v$ being constants,
	\begin{align*}
	T_e/T=o_p(1), \quad \mbox{as }  n\rightarrow \infty, T\rightarrow \infty.
	\end{align*}
\end{manualtheorem}
Theorem \ref{thm:clean} implies that as long as $T$ and $n$ both go to infinity, the rates do not matter for the purpose of proving label consistency. 

The next result  addresses the consistency for parameter estimation of $A$, which is based upon label consistency. First, we need a result that gives a faster decay rate of $T_e/T$ than Theorem \ref{thm:label_consistency}.
\begin{prop}\label{lemma:faster}
	Under the conditions of Theorem \ref{thm:separate}, if $(n_L^2 \log n_L)/T=o(1)$, $(n_L^3\log T)/(d^2 v T)=o(1) $ and $ (n_L^4 \log n_L )/ (d^4 v^2 n)=o(1) $,
	\begin{align*}
	n_L T_e/T=o_p(1), \quad \mbox{as } n_L\rightarrow \infty, n\rightarrow \infty, T\rightarrow \infty.
	\end{align*}
\end{prop}

\begin{thm}\label{thm:estimation}
	Under the conditions of Proposition \ref{lemma:faster}, if $ (n_L \log n)/T=o(1) $,
	\begin{align*}
	\max_{i\in \{1,...,n_L\}, j \in\{1,...,n\}} \left|\hat{A}_{ij}^{\hat{z}}-A_{ij} \right| =o_p(1), \quad \mbox{as } n_L \rightarrow \infty, n \rightarrow \infty, T \rightarrow \infty. 
	\end{align*}
\end{thm}
\begin{manualtheorem}{\ref{thm:estimation}$'$}
	Under the conditions of Theorem \ref{thm:separate_fixed}, if $(\log T)/(d^2 v T)=o(1) $, $d^4 v^2 n\rightarrow \infty $ and $ \log n/T=o(1) $,
	\begin{align*}
	\max_{i\in \{1,...,n_L\}, j \in\{1,...,n\}} \left|\hat{A}_{ij}^{\hat{z}}-A_{ij} \right| =o_p(1), \quad \mbox{as } n \rightarrow \infty, T \rightarrow \infty. 
	\end{align*}
\end{manualtheorem}

\section{THE HUB MODEL WITH THE NULL COMPONENT} \label{sec:null}
\subsection{Model setup}
In statistics, a null model generates data that match the basic features of the observed data, but which is otherwise a random process without structured patterns. In other words, a null model is the  degenerate case of the model class being studied. For example, the Erd\H{o}s-R\'{e}nyi random graph is the null model of stochastic block models (SBMs), i.e,  the SBM with only one community. The Newman--Girvan modularity \citep{Newman&Girvan2004} uses the configuration model as the null model in the criterion function for community detection. In regression analysis, a regression line with all regression coefficients being zero except the intercept can be viewed as a null model of multiple linear regression. 

The null model for grouped data, naturally, generates each group by independent Bernoulli trials. That is, if the grouping behavior is not governed by a network structure then every node is assumed to appear  independently in a group. The likelihood of $G^{(t)}$ under the null model is
\begin{equation*}
\mathbb{P}(G^{(t)})=\prod_{j=1}^n {\pi_{j}^{G_j^{(t)}} (1-\pi_{j})^{1-G_j^{(t)}}},
\end{equation*}
where $\pi_j$ is the probability that node $j$ appears in a group. 

The hub model studied in Section \ref{sec:classical} needs generalization to accommodate the null model because if there is only one component in \eqref{likelihood_classical}, say, node $i$ is the only leader, the likelihood of $G^{(t)}$ becomes 
\begin{equation*}
\mathbb{P}(G^{(t)})=\prod_{j=1}^n {A_{ij}^{G_j^{(t)}} (1-A_{ij})^{1-G_j^{(t)}}},
\end{equation*}
which is not a proper null model because of the assumption that $A_{ii} \equiv 1$  but $\pi_i$ is between 0 and 1. 

To better allow the hub model to degenerate to the null model, we add a null component to the hub model.  This null component  allows groups without hubs and nodes independently appear in such groups. We call this model the \textit{hub model with a null component}. We use $z^{(t)}=0$ to represent a hubless group. Under the hub model with a null component, each group $G^{(t)}$ is independently generated by the following two steps:
\begin{enumerate}[label=(\roman*)]
	\item The hub is sampled from a multinomial trial with parameter $\rho=(\rho_0,\rho_1,...,\rho_{n_L})$, i.e., $\mathbb{P}(z^{(t)}=i)=\rho_i$, with $\sum_{i=0}^{n_L} \rho_i=1$.
	
	\item If $z^{(t)}=i \in \{1,...,n_L\}$, then node $j$ will appear in the group independently with probability $A_{ij}$, i.e., $	\mathbb{P}(G_j^{(t)}=1|z^{(t)}=i)=A_{ij}$. If $z^{(t)}=0$, each node will independently join the group with probability $\pi_j$. 
\end{enumerate}
Note that the above model degenerates to the null model when $\rho_0=1$. As before we assume $A_{ii}\equiv 1$ for $i=1,...,n_L$.  The parameters for the hub model with a null component are
\begin{align*}
\rho & =(\rho_0,\rho_1,...,\rho_{n_L}), \\
A_{(n_L+1)\times n} & = \begin{pmatrix}
\pi_{1} &  \pi_2 & \cdots & \cdots  & \cdots & \cdots & \pi_n \\
1 &  A_{12} & \cdots & A_{1, n_L}  & A_{1,n_L+1} & \cdots & A_{1,n}\\
A_{21} & 1 & \cdots & A_{2,n_L} & A_{2,n_L+1} & \cdots & A_{2,n} \\
\vdots & \vdots & \ddots & \vdots & \vdots & \ddots & \vdots  \\ 
A_{n_L,1} & A_{n_L,2} & \dots & 1 & A_{n_L,n_L+1} & \cdots & A_{n_L,n}
\end{pmatrix}.
\end{align*}
Here the row indices of $A$ start from 0, i.e., $A_{0j}\equiv \pi_j$ for $j=1,...,n$. We will use $A_{0j}$ and $\pi_j$ interchangeably below. 
For  simplicity of notation, we use the same notation such as $\rho$ and $A$ for both the asymmetric hub model and the hub model with a null component when the meaning is clear from context. 

The new model has an advantage in data analysis in addition to  the theoretical benefit. Grouped data usually contain a number of tiny groups such as singletons and doubletons. When fitting the asymmetric hub model to such a data set, one sometimes has to include these nodes into the hub set due to the one-hub restriction. For example, a singleton must be included in the hub set and at least one node of a doubleton must be included, no matter how infrequently they appear in the data set, which may result in an unnecessarily large hub set. In the new model, these small groups can be treated as  hubless groups and the corresponding nodes may be removed from the hub set. Therefore, the model complexity is  significantly reduced.

\subsection{Identifiability under the hub model with a null component}
We study the identifiability of $\rho$ and $A$ under the new model in the same sense as Definition \ref{def:identi}. The parameter space of the hub model with a null component is $\mathcal{P}=\{(\rho,A) | 0\leq \rho_i \leq 1, i=0,...,n_L; A_{ii} = 1,i=1,...,n_L; 0\leq A_{ij}\leq 1, i=0,...,n_L, j=1,...,n, i\neq j   \}$. The parameters $(\rho,A)$ are identifiable if $\mathbb{P}(g|\tilde{\rho},\tilde{A}) \neq \mathbb{P}(g|\rho,A) $ 
for all realizations $g$ and parameters $(\tilde{\rho},\tilde{A}) \in \mathcal{P}$, $(\tilde{\rho},\tilde{A}) \neq (\rho,A)$.

Identifiability under the new model is much more difficult to prove than the case of the asymmetric hub model due to the extra null component in the model. In particular,  there is no constraint such as $\pi_i=1$ on  parameters of the null component. The conditions for identifiability in the following theorem are; however, as natural as those in Theorem \ref{thm:hub_iden}.
\begin{thm}\label{thm:hub_iden_null}
	The parameters $(\rho,A)$ of the hub model with a null component are identifiable under the following conditions:
	\begin{enumerate}[label=(\roman*)]
		\item $0<\rho_i<1$, for $i=0,...,n_L$;
		\item $A_{ij}<1$, for $i=0,...,n_L, j=1,...,n, i\neq j$;
		\item for any $i=1,...,n_L$, $i'=1,...,n_L,i\neq i'$, there exists $k \in \{n_L+1,...,n\}$ such that $A_{ik} \neq A_{i'k}$;
		\item for any $i=1,...,n_L$, there exist $k\in \{n_L+1,...,n\}$ and $k' \in \{n_L+1,...,n\}$ such that $\pi_k \neq A_{ik}$ and $\pi_{k'} \neq A_{ik'}$.
	\end{enumerate}
\end{thm}

Conditions (\romannumeral 1) - (\romannumeral 3) are identical to Theorem \ref{thm:hub_iden}. Condition (\romannumeral 4), adds the requirement that for any hub $i$, there exist two followers  which each have different probabilities of appearing in a group led by  hub $i$  than of appearing in a  hubless group. This condition implies that there should exist at least two more nodes in the  node set than in the hub set. This condition is  natural if one compares it to condition (\romannumeral 3) in Theorem \ref{thm:hub_iden}, as both imply that there exists at least one more column than rows in $A$.

The proof of Theorem \ref{thm:hub_iden_null} is given in the Appendix.

\subsection{Consistency of the maximum profile likelihood estimator for the hub model with a null component}\label{sec:null_consistency}

We  establish the consistency for the new model in the same setting as in Section \ref{sec:classical_consistency}. That is, we let $n_L\rightarrow \infty, n\rightarrow \infty, T\rightarrow \infty$ and treat  $z_*$ as a random vector. This proof is more challenging than the proof of consistency for  the asymmetric hub model due to the extra null component.

We adopt the notation in Section \ref{sec:classical_consistency}. Let $t_{i*}$ be the true number of groups  initiated by hub node $i$  and let $t_{0*}$ be the true number of  hubless groups, i.e., $t_{i*}=\sum_t 1(z_*^{(t)}=i)$, for $i=0,...,n_L$. We give bounds for $t_{i*}$ similarly to Lemma \ref{thm:rho}. The proof is omitted due to this similarity.
\begin{lemma}
	Assume $\min\{\rho_0,\rho_1,...,\rho_{n_L}\}=c_{\textnormal{min}*} /n_L$ and $\max\{\rho_0,\rho_1,...,\rho_{n_L}\}=c_{\textnormal{max}*} /n_L$ where $c_{\textnormal{min}}$ and $c_{\textnormal{max}}$ are positive constants. When $n_L^2 (\log n_L)/T=o(1)$, there exist positive constants $c_\textnormal{min}$ and $c_{\textnormal{max}}$ such that 
	\begin{align*}
	\frac{Tc_{\textnormal{min}}}{n_L} \leq t_{i*} \leq \frac{Tc_{\textnormal{max}}}{n_L}, \quad i=0,...,n_L,
	\end{align*} 
	with probability approaching 1. 
\end{lemma}

Let $z=(z^{(t)})_{t=1,...,T}$ be an arbitrary assignment of hub labels, where $z^{(t)} \in \{0,1,...,n_L \}$. Given $z$, $\hat{A}_{ij}$, $\bar{A}_{ij}$, $L_G(z,A)$, $L_P(z,A)$, $L_G(z)$ and $L_P(z)$ are defined as before. As in Section \ref{sec:classical_consistency}, we first bound the difference between $L_G(z)$ and $L_P(z)$. 

\begin{thm}\label{thm:uniform0}
	For all $\eta>0 $,
	\begin{align*}
	\mathbb{P}(\max_z |L_G(z)-L_P(z)|\geq 2\eta ) \leq & (n_L+1)^T (T/(n_L+1)+1)^{(n_L+1) n} e^{-\eta} \\
	& +4 (n_L+1)^T \exp \left \{ -\frac{\eta^2}{4(2nT+\eta)} \right \}.
	\end{align*}
\end{thm}
The proof is omitted as it is similar to Theorem \ref{thm:uniform} with minor modifications. 

We now  prove a result on the separation of $L_P(z_*)$ from $L_P(\hat{z})$ which is similar to Theorem \ref{thm:separate}. However, the technique in the original proof cannot be directly applied to the new model. A key step in the proof of Theorem \ref{thm:separate} relies on the fact that we can obtain a non-zero lower bound for the number of correctly classified groups with node $i$ as the hub node in the asymmetric hub model. Specifically, let $t_{ii}$ be the number of correctly classified groups  where node $i$ is the hub node. For  the asymmetric hub model, we obtained a lower bound for $t_{ii}/t_{i*}\,\, (i=1,...,n_L)$ in the proof of Theorem \ref{thm:separate} from the fact that a node cannot be  labeled as the  hub of a particular group if the node does not appear in the group. This is due to the assumption $A_{ii} \equiv 1$ for $i=1,...,n_L$. For the hub model with a null component,  the lower bound for $t_{ii}/t_{i*}$ cannot be proved by the same technique  because  all groups can be  classified as  hubless groups without violating the assumption $A_{ii}\equiv 1$. 

We take a different path in the proof to overcome this issue and other technical difficulties due to the null component. The proofs of the theorems in this section are given in the Appendix. We first bound $t_{i0}/t_{i*}$ for $i=1,...,n$. 
\begin{thm}\label{thm:ti0}
	Assume
	\begin{enumerate}[label=(\roman*)]
		\item there exists a set $V_i \subset \{1,...,n\}$ for $i=1,...,n_L$ such that $|V_i| \geq v n/n_L$ and $A_{ij}-A_{i'j} \geq d$ for all $j \in V_i$, $i\neq i'$;
		\item  $A_{ii'}\leq c_0/n_L $ for $i=0,...,n_L$, $i'=1,...,n_L$, $i\neq i'$, where $c_0$ is a positive constant.
	\end{enumerate}
	Then if $(n_L^2 \log n_L)/T=o(1)$, $ (n_L^5\log T)/(d^2 v T)=o(1) $ and $ (n_L^8 \log n_L )/ (d^4 v^2 n)=o(1) $, for all $\eta>0$,
	\begin{align*}
	\frac{t_{i0}}{t_{i*}}\leq\eta, \quad i=1,...,n_L,
	\end{align*}
	with probability approaching 1.
\end{thm}

Based on the result in Theorem \ref{thm:ti0}, we establish the label consistency for the hub model with a null component.

\begin{thm}\label{thm:label_consistency_null}
	Under the conditions of Theorem \ref{thm:ti0}, 
	\begin{align*}
	\frac{T_e}{T} =o_p(1), \quad \mbox{as } n_L\rightarrow \infty, n\rightarrow \infty, T\rightarrow \infty.
	\end{align*}
\end{thm}

\begin{manualtheorem}{\ref{thm:label_consistency_null}$'$}\label{thm:label_null_fixed}
	Assume
	\begin{enumerate}[label=(\roman*)]
		\item there exists a set $V_i \subset \{1,...,n\}$ for $i=1,...,n_L$ such that $|V_i| \geq v n/n_L$ and $A_{ij}-A_{i'j} \geq d$ for all $j \in V_i$, $i\neq i'$;
		\item  $A_{ii'}$ is bounded away from 1 for $i=0,...,n_L$, $i=1,...,n_L$, $i\neq i'$.
	\end{enumerate}
	If $ (\log T)/(d^2 v T)=o(1) $ and $d^4 v^2 n\rightarrow \infty $, then
	\begin{align*}
	\frac{T_e}{T} =o_p(1), \quad \mbox{as } n\rightarrow \infty, T\rightarrow \infty.
	\end{align*}	
\end{manualtheorem}

We conclude this section by the result on consistency for parameter estimation of $A$. 
\begin{thm}\label{thm:estimation_null}
	Under the conditions of Theorem \ref{thm:ti0}, if $(n_L^2 \log n_L)/T=o(1)$, $ (n_L^5\log T)/(d^2 v T)=o(1) $, $ (n_L^8 \log n_L )/ (d^4 v^2 n)=o(1)$ and $(n_L\log n)/T=o(1)$,
	\begin{align*}
	\max_{i\in \{0,...,n_L\}, j \in\{1,...,n\}} \left|\hat{A}_{ij}^{\hat{z}}-A_{ij} \right| =o_p(1), \quad \mbox{as } n_L \rightarrow \infty, n \rightarrow \infty, T \rightarrow \infty. 
	\end{align*}
\end{thm}
\begin{manualtheorem}{\ref{thm:estimation_null}$'$}
	Under the conditions of Theorem \ref{thm:label_null_fixed}, if $ (\log T)/(d^2 v T)=o(1) $, $d^4 v^2 n\rightarrow \infty $ and $\log n/T=o(1)$,
	\begin{align*}
	\max_{i\in \{0,...,n_L\}, j \in\{1,...,n\}} \left|\hat{A}_{ij}^{\hat{z}}-A_{ij} \right| =o_p(1), \quad \mbox{as } n \rightarrow \infty, T \rightarrow \infty. 
	\end{align*}
\end{manualtheorem}

\section{NUMERICAL STUDIES}\label{sec:numerical}

We examine the performance of the estimators for the  asymmetric hub model and the hub model with a null component, especially the behavior of the estimators with varying $n_L$, $n$ and $T$, by simulation studies. 

We first introduce the simulation setups for the asymmetric hub model. Let the size of the hub set, $n_L$, be 10 or 20. Each node $i$ within the hub set is to be selected as a group hub with the probability $\rho_i= 1/ n_L$. Let the size of the node set, $n$, be 100 or 1000. 

We partition the follower set $\{n_L+1,...,n\}$ into $n_L$ non-overlapping sets $V_1,...,V_{n_L}$, i.e., $\{n_L+1,...,n\}=V_1\cup \cdots \cup V_{n_L}$ and $V_i \cap V_{i'}=\emptyset$ for $i\neq i'$. Each set $V_i$ is the set of followers with a preference for hub $i$ over other hub nodes. As in Theorem \ref{thm:separate} and \ref{thm:separate_fixed}, we assume different ranges of probabilities of joining a group for   followers that prefer a specific leader than for followers which do not prefer that leader. That is, for $j \in V_i$, the parameters $A_{ij}$ are generated independently with $U(0.2,0.4)$, and for $j \notin V_i$, the parameters $A_{ij}$ are generated independently with $U(0,0.2)$. 

For clarification, we will not use  prior information about how $A$ was generated  in the estimating procedure. That is, we still treat $A$ as unknown fixed parameters in the estimation. We generate these probabilities from uniform distributions for the sole purpose of adding more variations to the parameter setup. 

In each setup, we consider three different sample sizes, $T=1000, 1500$ and 2000.

For the hub model with a null component, let the probability of hubless groups $\rho_0=0.2$, and let $\rho_i=0.8/n_L$ for $i=1,...,n_L$. For a  hubless group, each node will independently join the group with probability $\pi_j \equiv 0.05$ for $j=1,...,n$. The setups on $n_L$, $n$, $\{V_1,...,V_{n_L} \}$, $A$ and $T$ are identical to the asymmetric hub model case. 

Exact maximization in $z$ of $L_G(z)$ for both the asymmetric hub model and the hub model with a null component is computational intractable. However, a classical algorithm -- the hard expectation-maximization (EM) algorithm --  can be used in our maximization problem. This is different from the case of stochastic block models because given $A$, the assignment of $z^{(t)}$ for each group does not interact with the labels of other groups. 

Below we give the hard EM algorithm for the asymmetric hub model. The algorithm for the hub model with a null component is almost identical with minor modifications. 

\textbf{Algorithm 1:} (Hard EM)
We iteratively update $\hat{A}$ and $\hat{z}$ by the following the M-steps and E-steps until convergence. 
\begin{itemize}
	\item[] \textbf{M-step:} Given $\hat{z}$, update $\hat{A}$  by
	\begin{align*}
	\hat{A}_{ij}=\frac{\sum_t G_j^{(t)} 1(\hat{z}^{(t)}=i) }{\sum_t  1(\hat{z}^{(t)}=i)}, \,\, \textnormal{for } i=1,...,n_L, j=1,...,n.
	\end{align*}
	
	\item[] \textbf{E-step:} Given $\hat{A}$, update $\hat{z}$ by 
	\begin{align*}
	\hat{z}^{(t)}=\argmax_{\{1,...,n \}} \sum_j G_j^{(t)} \log \hat{A}_{ij}+  (1-G_j^{(t)}) \log (1-\hat{A}_{ij}).
	\end{align*}
\end{itemize}

The proposed algorithm is in line with the common usage of hard EM in the literature, such as the $k$-means algorithm \citep{macqueen1967}, which is a variant of the (soft) EM algorithm for Gaussian mixture models \citep{Fraley2002}. To help guard against local maxima, we run each algorithm  for 20 random initial values.

\begin{table}[!ht]
	\caption{Asymmetric hub model results. Mis-labels: the fraction of groups with incorrect hub labels.  RMSE($\hat{A}_{ij}$): average RMSEs when the hub labels are unknown.  RMSE*: average RMSEs when the hub labels are known.   Standard deviations $\times 10^4$ in parentheses.}
	\label{Table:simulation1}
	\begin{center}
		\begin{tabular}{|c | c | c |c|c|c|c|}
			\hline
			
			$n_L=10$	& \multicolumn{3}{c|}{$n=100$} & \multicolumn{3}{c|}{$n=1000$}\\
			
			\hline
			& Mis-labels & RMSE($\hat{A}_{ij}$) & RMSE*    & Mis-labels & RMSE($\hat{A}_{ij}$)  & RMSE*    \\
			$T=1000$ &	0.0379	(74) & 0.0327	(9) & 0.0314 	(8)	&	0.0000	(0) &	0.0315	(3) &	0.0315	(3) \\
			$T=1500$ &	0.0311	(56) & 0.0266	(7) & 0.0256 	(7) &	0.0000	(0) &	0.0257	(2) &	0.0257	(2) \\
			$T=2000$ &	0.0283	(50) & 0.0229	(6) & 0.0222 	(6)	&	0.0000	(0) &	0.0223	(2) &	0.0223	(2)  \\
			
			\hline

			$n_L=20$	& \multicolumn{3}{c|}{$n=100$} & \multicolumn{3}{c|}{$n=1000$}\\
			
			\hline
			& Mis-labels & RMSE($\hat{A}_{ij}$) & RMSE*    & Mis-labels & RMSE($\hat{A}_{ij}$)  & RMSE*    \\
			$T=1000$  &	0.1885	(171) &	0.0513	(13) & 0.0432	(9)	&	0.0016	(14) &	0.0435	(3)  &	0.0434	(3)\\
			$T=1500$ &	0.1396	(129) &	0.0408	(10) & 0.0351	(7) &	0.0003	(5)&	0.0354	(2)  &	0.0353	(2) \\
			$T=2000$ &	0.1159	(103)  &	0.0347	(8)  & 0.0304	(6)	&	0.0001	(2) &	0.0306	(2) &	0.0306	(2)\\

			\hline
			
		\end{tabular}
	\end{center}
	
\end{table}

\begin{table}[!ht]
	\caption{Hub model with null component results. Mis-labels: the fraction of groups with incorrect hub labels.  RMSE($\hat{A}_{ij}$): average RMSEs when the hub labels are unknown.  RMSE*: average RMSEs when the hub labels are known.   Standard deviations $\times 10^4$ in parentheses.}
	\label{Table:simulation2}
	\begin{center}
		\begin{tabular}{|c | c | c |c|c|c|c|}
			\hline
			
			$n_L=10$	& \multicolumn{3}{c|}{$n=100$} & \multicolumn{3}{c|}{$n=1000$}\\
			
			\hline
			& Mis-labels & RMSE($\hat{A}_{ij}$) & RMSE*    & Mis-labels & RMSE($\hat{A}_{ij}$)  & RMSE*    \\
			$T=1000$  &	0.0739	(104) &	0.0364	(12) & 0.0353	(11)   &	0.0017	(28) & 0.0353 (5) & 0.0353	(5)	          \\
			$T=1500$ &	0.0638	(79) &	0.0296	(9) & 0.0287	(8)    & 0.0001	(6)  	& 0.0288  (3) & 0.0288	(3)          \\
			$T=2000$ &	0.0588	(66) &	0.0256	(8) & 0.0248	(7)   &	0.0000	(1) &	0.0249	(3) & 0.0249	(3)   
			\\
			
			\hline

			$n_L=20$	& \multicolumn{3}{c|}{$n=100$} & \multicolumn{3}{c|}{$n=1000$}\\
			
			\hline
			& Mis-labels & RMSE($\hat{A}_{ij}$) & RMSE*    & Mis-labels & RMSE($\hat{A}_{ij}$)  & RMSE*    \\
			$T=1000$ &	0.2620	(186) 	& 0.0561 	(15) & 0.0484	(10) &	0.0127	(69)	 &	0.0484	(6) & 0.0487	(5) \\
			$T=1500$ &	0.2138	(142) 	& 0.0451	(11) & 0.0393	(8)	 &	0.0014	(16) &	0.0396	(4) & 	0.0396	(4)  \\
			$T=2000$ &	0.1865	(120) 	& 0.0385	(9) & 0.034	(7)		 &	0.0003	(4)  &	0.0342	(3) & 0.0342	(3) \\

			\hline
			
		\end{tabular}
		
	\end{center}
	
\end{table}

Table \ref{Table:simulation1} and \ref{Table:simulation2} show the performance of the estimators for the asymmetric hub model and the hub model with the null component, respectively. The first measure of performance we are interested in is the proportion of mislabeled groups, $T_e/T$.  As the proportion of mislabeled groups approaches zero, we expect the parameter estimates to approach the accuracy achievable if the hub nodes are known.  The second measure of performance is the RMSE($\hat{A}_{ij}$).  As a reference point, we also provide the RMSE achieved when we treat the hub nodes as known, RMSE*.  All results are based on 500 replicates.

From the tables, the estimators for the asymmetric hub model generally outperform those for the hub model with the null component as the latter is a more complex model. The patterns within the two tables are; however, similar. First, the performance becomes better as the sample size $T$ grows, which is in line with common sense in statistics. Second, the performance becomes worse as $n_L$ grows because $n_L$ is the number of components in the mixture model, and thus a larger $n_L$ indicates a more complex model. Third, the effect of $n$ is slightly more complicated: the RMSE* for the case that hub labels are known slightly increases as $n$ grows because the model contains more parameters. What we are interested in is the case where hub labels are unknown, and this is what our theoretical studies focused on. In this case, the RMSE($\hat{A}_{ij}$) significantly improves as $n$ grows. This is because the clustered pattern becomes clearer as the number of followers increases, which is in line with the label consistency results in Section \ref{sec:classical_consistency} and \ref{sec:null_consistency}.

\section{SUMMARY AND DISCUSSION} \label{sec:summary}
In this paper we studied the theoretical properties of the hub model and its variant from the perspective of Bernoulli mixture models. 
The contribution of the paper is three-fold. First, we proved the model identifiability  of the hub model. Bernoulli mixture models are a notoriously difficult model to prove identifiability on, especially under mild conditions. Second, we proved the label consistency and estimation consistency of the hub model. Third, we generalized the hub model by adding a null component that allows nodes to independently appear in hubless groups. The new model can naturally degenerate to the null model -- the Bernoulli product. We also proved identifiability and consistency of the newly proposed model. 

A natural constraint from \cite{Zhao2015} which we  apply to the hub model is $A_{ii}=1 \,\, (i=1,...,n_L)$, which turns out to be a key condition for ensuring model identifiability and  avoiding the label swapping issue in the proof of consistency. On the other hand, this constraint prevents the asymmetric hub model from naturally degenerating to the null model because one node always appear in every group when there is only one component in the hub model, which motivated adding the null component to the model. 

For future works, we plan to study the problem of model selection  -- in the context of the hub model, the problem of how to identify the hub set from the grouped data. This is seemingly a more difficult problem than estimating the number of communities in the context of community detection because not only do we need to estimate the size of the hub set, $n_L$, but we must also estimate which nodes belong to the hub set. Another direction we would like to explore is to go beyond the independence assumption and to develop theory and model selection methodologies for correlated or temporal-dependent groups \citep{zhao2018temporal}.

\section*{Appendix}
We start by recalling notations defined in the main text. Recall that $z_{*}$ is the true label assignment, $z$ is an arbitrary label assignment, and $\hat{z}$ is the maximum profile likelihood estimator. Furthermore, $t_{i*}=\sum_t 1(z_*^{(t)}=i)$, and $ t_{i}=\sum_t 1(z^{(t)}=i),t_{ii'}=\sum_t  1(z_*^{(t)}=i, \hat{z}^{(t)}=i' )$.

\begin{proof}[Proof of Lemma \ref{thm:rho}]
	By Hoeffding's inequality \cite{Hoeffding63}, for  sufficiently small $\epsilon$,
	\begin{align}
	\mathbb{P} \left ( \left |\frac{t_{i*}}{T}-\rho_i \right | \geq \frac{\epsilon}{n_L} \right ) \leq 2 \exp \left \{ - \frac{2T \epsilon^2}{n_L^2} \right \}.
	\end{align}
	So
	\begin{align*}
	\mathbb{P} \left ( \left |\frac{t_{i*}}{T}-\rho_i \right |\leq \frac{\epsilon}{n_L}, \,\, i=1,...,n_L  \right ) \geq 1- 2 n_L  \exp \left \{ - \frac{2T \epsilon^2}{n_L^2} \right \} \rightarrow 1.
	\end{align*}
	The conclusion holds by letting $c_{\textnormal{max}}=c_{\textnormal{max}*}+\epsilon$ and $c_{\textnormal{min}}=c_{\textnormal{min}*}-\epsilon$.
\end{proof}

\begin{proof}[Proof of Lemma \ref{thm:partition}]
	\begin{align*}
	L_G(z)-L_P(z)  = &  \left ( \sum_{i=1}^{n_L} t_i \sum_{j } \hat{A}_{ij} \log \hat{A}_{ij}+(1-\hat{A}_{ij}) \log (1-\hat{A}_{ij}) \right ) \\
	& -  \left ( \sum_{i=1}^{n_L} t_i \sum_{j } \hat{A}_{ij} \log \bar{A}_{ij}+(1-\hat{A}_{ij}) \log (1-\bar{A}_{ij})  \right ) \\
	& + \left ( \sum_{i=1}^{n_L} t_i \sum_{j } \hat{A}_{ij} \log \bar{A}_{ij}+(1-\hat{A}_{ij}) \log (1-\bar{A}_{ij})  \right ) \\
	& - \left ( \sum_{i=1}^{n_L} t_i \sum_{j } \bar{A}_{ij} \log \bar{A}_{ij}+(1-\bar{A}_{ij}) \log (1-\bar{A}_{ij})  \right ) \\
	= & \sum_{i=1}^{n_L} t_i \sum_{j } D(\hat{A}_{ij}|\bar{A}_{ij}) +B_{n_L,n,T}.
	\end{align*}
\end{proof}

\begin{proof}[Proof of Theorem \ref{thm:uniform}]
	Due to Lemma \ref{thm:partition}, we bound the two parts  separately. First consider $\sum_{i=1}^{n_L} t_i \sum_{j} D(\hat{A}_{ij}|\bar{A}_{ij}) $. 
	We adopt an inequality proved in \cite{Choietal2011}, which is based on a heterogeneous Chernoff bound in \cite{dubhashi2009concentration}. 
	Let $\nu$ be any realization of $\hat{A}$. 	
	\begin{align*}
	\mathbb{P}(\hat{A}_{ij}=\nu_{ij}|z_{*})\leq e^{-t_i D(\nu_{ij} | \bar{A}_{ij})}.
	\end{align*}
	By the independence of $\hat{A}_{ij}$ conditional on $z_{*}$,
	\begin{align*}
	\mathbb{P}(\hat{A}=\nu|z_*) \leq \exp \left \{ -\sum_{i=1}^{n_L} \sum_{j} t_i D( \nu_{ij} |\bar{A}_{ij}) \right \}.
	\end{align*}
	Let $\hat{\mathcal{A}}$ be the range of $\hat{A}$ for a fixed $z$. Then $|\hat{\mathcal{A}} |\leq \prod_{i=1}^{n_L} (t_i+1)^n \leq  \prod_{i=1}^{n_L} (t_i+1)^n  \leq (T/n_L+1)^{n_L n}$, as $\hat{A}_{ij}$ can only take values from $0/t_i,1/t_i,...,t_i/t_i$.
	
	For all $\eta>0$,
	\begin{align*}
	& \mathbb{P} \left ( \left . \sum_{i=1}^{n_L} \sum_{j } t_i D(\hat{A}_{ij} |\bar{A}_{ij}) \geq \eta  \right | z_* \right ) \\
	=& \sum_{\nu \in \hat{\mathcal{A}}} \mathbb{P} \left ( \left . \hat{A}=\nu, \sum_{i=1}^{n_L} \sum_{j } t_i D(\nu_{ij} |\bar{A}_{ij}) \geq \eta \right | z_* \right ) \\
	\leq & \sum_{\nu \in \hat{\mathcal{A}}} \exp \left \{ -\sum_{i=1}^{n_L} \sum_{j } t_i D( 
	\nu_{ij} |\bar{A}_{ij}) \right \} 1 \left \{  -\sum_{i=1}^{n_L} \sum_{j } t_i D(\nu_{ij} |\bar{A}_{ij}) \leq - \eta \right \} \\
	\leq  & \sum_{\nu \in \hat{\mathcal{A}}} e^{-\eta} \leq  | \hat{\mathcal{A}} |  e^{-\eta} \leq  (T/n_L+1)^{n_L n} e^{-\eta}, 
	\end{align*}
	and
	\begin{align*}
	\mathbb{P} \left (  \max_{z} \sum_{i=1}^{n_L} \sum_{j } t_i D( \hat{A} |\bar{A}_{ij}) \geq \eta \right ) \leq n_L^T  (T/n_L+1)^{n_L n} e^{-\eta}.
	\end{align*}
	
	Now we bound the second term. Note that the second term is 
	\begin{align*}
	B_{n_L,n,T}= & \sum_{i=1}^{n_L} t_i \left ( \sum_{j } (\hat{A}_{ij} \log \bar{A}_{ij}+(1-\hat{A}_{ij}) \log (1-\bar{A}_{ij})) \right.  \\ 
	& -\left. \sum_{j }( \bar{A}_{ij} \log \bar{A}_{ij}+(1-\bar{A}_{ij}) \log (1-\bar{A}_{ij})  ) \right ).
	\end{align*}
	This term can be bounded by Theorem 2 in \cite{zhao_bernstein} with a slight generalization. Let $Y_j=\sum_{i=1}^{n_L} t_i   (\hat{A}_{ij} \log \bar{A}_{ij}-\bar{A}_{ij} \log \bar{A}_{ij})   $. From the proof of Theorem 2 in \cite{zhao_bernstein}, for any $|\lambda|<1$, the moment generating function of $Y_j$, $	\mathbb{E} [e^{\lambda Y_j }|z_*]\leq \exp \left \{ \frac{T \lambda^2}{2(1-|\lambda|)}  \right  \}$. Therefore, $	\mathbb{E} [e^{\lambda \sum_{j}Y_j }|z_*]\leq \exp \left \{ \frac{nT \lambda^2}{2(1-|\lambda|)}  \right  \} $. By the same argument in \cite{zhao_bernstein} using the Chernoff bound,
	\begin{align*}
	\mathbb{P} \left ( \left .  |B_{n_L,n,T}| \geq \eta \right |z_* \right ) \leq 4 \exp \left \{ -\frac{\eta^2}{4(2nT+\eta)} \right \}.
	\end{align*}
	It follows that
	\begin{align*}
	\mathbb{P} \left ( \max_{z} |B_{n_L,n,T}| \geq \eta \right ) \leq 4 n_L^T \exp \left \{ -\frac{\eta^2}{4(2nT+\eta)} \right \}.
	\end{align*}
\end{proof}

\begin{proof}[Proof of Theorem \ref{thm:separate}]
	Recall that $t_{ik}=\sum_t  1(z_*^{(t)}=i, \hat{z}^{(t)}=k ), t_{i}=\sum_t 1(\hat{z}^{(t)}=i), i=1,...,n_0, k=1,...,n_0$. Here we slight abuse the notation $t_i$ as it is defined specifically for the MLE $\hat{z}$ not for an arbitrary $z$. 
	
	We first prove a fact: if $(n_0^2 \log n_0)/T=o(1)$, for $0<\delta_1<e^{-c_0}$, 
	\begin{align*}
	\mathbb{P} \left (  \bigcup_{i=1}^{n_0} \left  \{\frac{t_{ii}}{t_{i*}}\leq \delta_1 \right \} \right ) \rightarrow 0.
	\end{align*}
	To prove it, note that since $\hat{z}$ must be feasible (the estimated hub must appear in the group as we assume $A_{ii}\equiv 1$), we have 
	\begin{align}
	& \mathbb{P} \left ( \left . \frac{t_{ii}}{t_{i*}}\leq \delta_1 \right | z_* \right ) \nonumber \\
	\leq & \mathbb{P} \left ( \left . \frac{1}{t_{i*}}  \sum_{t=1}^T 1(z_*^{(t)}=i)  \prod_{k \in \{1,...,n_0 \}, k\neq i} (1-G_k^{(t)}) \leq \delta_1 \right | z_* \right ) \label{bern2}.
	\end{align}
	Now since 
	\begin{align*}
	\mathbb{E} \left [ \left .  \prod_{k \in \{1,...,n_0 \}, k\neq i}  (1-G_k^{(t)}) \right | z_*^{(t)}=i \right ] = \prod_{k \in \{1,...,n_0 \}, k\neq i} (1-A_{ik})   \geq (1-c_0/n_0)^{n_0} \geq e^{-c_0},
	\end{align*}
	by Hoeffding's inequality,
	\begin{align*}
	\eqref{bern2}  \leq & \mathbb{P} \left ( \left . \frac{1}{t_{i*}}  \sum_{t=1}^T 1(z_*^{(t)}=i) \left [ \prod_{k \in \{1,...,n_0 \}, k\neq i} (1-G_k^{(t)}) -\prod_{k \in \{1,...,n_0 \}, k\neq i} (1-A_{ik})\right ] \leq \delta_1-e^{-c_0} \right | z_* \right ) \\
	\leq & \exp \{ -2t_{i*} (e^{-c_0}-\delta_1 )^2 \}.
	\end{align*}
	Hence
	\begin{align*}
	& \mathbb{P} \left ( \left .  \bigcup_{i=1}^{n_0} \left \{\frac{t_{ii}}{t_{i*}}\leq \delta_1 \right \} \right | z_* \right ) \\
	= & \mathbb{P} \left ( \left .   \bigcup_{i=1}^{n_0} \left \{\frac{t_{ii}}{t_{i*}}\leq \delta_1 \right \}, \{ t_{i*} \geq c_\textnormal{min} T/n_0, \textnormal{for all $i$} \} \right | z_* \right ) \\
	& + \mathbb{P} \left ( \left .   \bigcup_{i=1}^{n_0} \left \{\frac{t_{ii}}{t_{i*}}\leq \delta_1 \right \}, \{ t_{i*} < c_\textnormal{min} T/n_0, \textnormal{for some $i$} \} \right | z_* \right ) \\
	\leq  &  \sum_{i=1}^{n_0} \mathbb{P} \left ( \left . \frac{t_{ii}}{t_{i*}}\leq \delta_1 \right | z_* \right )1(t_{i*}\geq c_\textnormal{min} T/n_0) \\
	& +1(t_{i*} < c_\textnormal{min} T/n_0, \textnormal{for some $i$}) \\
	\leq &  n_0 \exp \{ -2c_\textnormal{min} T/(n_0) (e^{-c_0}-\delta_1 )^2 \} +1(t_{i*} < c_\textnormal{min} T/n_0, \textnormal{for some $i$}).
	\end{align*}
	It follows that
	\begin{align*}
	& \mathbb{P} \left (   \bigcup_{i=1}^{n_0} \left \{\frac{t_{ii}}{t_{i*}}\leq \delta_1 \right \} \right ) \\
	= & \mathbb{E}_{z_*} \left [\mathbb{P} \left ( \left .  \bigcup_{i=1}^{n_0} \left \{\frac{t_{ii}}{t_{i*}}\leq \delta_1 \right \} \right | z_* \right )\right ] \\
	= & n_0 \exp \{ -2c_\textnormal{min} T/(n_0) (e^{-c_0}-\delta_1 )^2 \} +\mathbb{P}(t_{i*} < c_\textnormal{min} T/n_0, \textnormal{for some $i$}) \rightarrow 0.
	\end{align*}
	Therefore, $\frac{t_{ii}}{t_{i*}}\geq \delta_1 $ for $i=1,...,n_0$ with probability approaching 1.
	
	Let $\mathcal{E}=\{ \frac{t_{ii}}{t_{i*}}\geq \delta_1 \textnormal{ and } t_{i*}\geq c_\textnormal{min} T/n_0, i=1,...,n_0  \}$. We have shown $\mathbb{P}(\mathcal{E})\rightarrow 1$. 
	The inequalities below are proved within the set $\mathcal{E}$, and thus hold with probability approaching 1. 
	
	For $i=1,...,n_0$, $k=1,...,n_0$, $k\neq i$,
	\begin{align*}
	\frac{t_{ik}}{t_k} =\frac{t_{ik}}{\sum_{k'=1}^{n_0} t_{k'k}} \leq \frac{t_{ik}}{t_{ik}+t_{kk}} =\frac{t_{ik}/t_{i*}}{t_{ik}/t_{i*} +t_{kk}/t_{k*} \cdot t_{k*}/t_{i*}} \leq \frac{1}{1+\delta_1\cdot c_{\textnormal{max}}/c_{\textnormal{min}}} =\delta_2.
	\end{align*}
	Now we give a lower bound for $A_{ij}-\bar{A}_{kj}$ for $j \in V_i$ and $k \neq i$,
	\begin{align*}
	A_{ij}-\bar{A}_{kj} & = \frac{\sum_t (A_{ij}-P_j^{(t)}) 1(\hat{z}^{(t)}=k)}{t_k} \\
	& = \frac{\sum_{k'=1}^{n_0} (A_{ij}-A_{k'j}) t_{k'k}}{t_k} \\
	& \geq \frac{d \sum_{k' \neq i} t_{k'k}}{t_k}  \geq	d (1-\delta_2).
	\end{align*}
	Finally,
	\begin{align*}
	L_P(z_*)-L_P(\hat{z}) & \geq \sum_t \sum_j 2(P_j^{(t)}-\bar{A}_{\hat{z}^{(t)},j})^2 \\
	& \geq \sum_{i=1}^{n_0} \sum_{k\neq i} \sum_{t: z_*^{(t)}=i, \hat{z}^{(t)}=k}  \sum_{j \in V_i} 2(A_{ij}-\bar{A}_{kj})^2 \\
	& \geq \sum_{i=1}^{n_0} \sum_{k\neq i} \sum_{t: z_*^{(t)}=i, \hat{z}^{(t)}=k} \sum_{j \in V_i} 2d^2 (1-\delta_2)^2 \\
	& \geq 2d^2 (1-\delta_2)^2 vn/n_0 T_e,
	\end{align*}
	where the first inequality follows from a basic inequality for Kullback–Leibler divergence (see \cite{popescu2016bounds}).
\end{proof}

\begin{proof}[Proof of Theorem \ref{thm:label_consistency}]
	First we prove 
	\begin{align*}
	\max_z \frac{  n_L }{d^2 v nT}|L_G(z)-L_P(z)| =o_p(1).
	\end{align*}
	From Theorem \ref{thm:uniform},
	\begin{align*}
	\mathbb{P}(\max_z |L_G(z)-L_P(z)|\geq 2 \eta ) \leq n_L^T (T/n_L+1)^{n_L n} e^{-\eta}+4 n_L^T \exp \left \{ -\frac{\eta^2}{4(2nT+\eta)} \right \}.
	\end{align*}
	The above bound goes 0 if $\eta=(d^2 v nT \epsilon) /n_L$,  $ (n_L^2\log T)/(d^2 v T)=o(1) $ and $ (n_L^2 \log n_L )/ (d^4 v^2 n)=o(1) $. 
	
	Now for all $\epsilon>0$, 
	\begin{align*}
	& \mathbb{P} \left ( \frac{T_e}{T} \geq \epsilon  \right ) \\
	= & \mathbb{P} \left ( \frac{T_e}{T} \geq \epsilon, \frac{\delta  n_L }{d^2 v nT}(L_P(z_*)-L_P(\hat{z})) \geq \frac{T_e}{T} \right ) \\
	&  +  \mathbb{P} \left ( \frac{T_e}{T} \geq \epsilon,  \frac{\delta  n_L }{d^2 v nT}(L_P(z_*)-L_P(\hat{z})) < \frac{T_e}{T}  \right ) \\
	= & \mathbb{P} \left (  \frac{\delta  n_L }{d^2 v nT}(L_P(z_*)-L_P(\hat{z}))  \geq \epsilon \right )+ o(1) \quad \textnormal{(by Theorem \ref{thm:separate})} \\
	= & \mathbb{P} \left (  \frac{\delta  n_L }{d^2 v nT} \left ( (L_P(z_*)-L_G(z_*))+ (L_G(z_*)-L_G(\hat{z}))+(L_G(\hat{z})-L_P(\hat{z})) \right ) \geq \epsilon \right )+ o(1) \\
	\leq & \mathbb{P} \left (  \frac{\delta  n_L }{d^2 v nT} \left ( |L_P(z_*)-L_G(z_*)|+|L_G(\hat{z})-L_P(\hat{z})| \right ) \geq \epsilon \right )+ o(1) \\
	\rightarrow & 0 \quad \textnormal{(by Theorem \ref{thm:uniform})}.
	\end{align*}
\end{proof}

\begin{proof}[Proof of Lemma \ref{lemma:faster}]
	It is sufficient to show 	
	\begin{align*}
	\max_z \frac{  n_L^2 }{d^2 v nT}|L_G(z)-L_P(z)| =o_p(1),
	\end{align*}
	which holds under the conditions of the lemma by the same argument in the proof of Theorem \ref{thm:label_consistency}. 
\end{proof}

\begin{proof}[Proof of Theorem \ref{thm:estimation}]
	First we bound $\left |\hat{A}_{ij}^{\hat{z}}-\hat{A}_{ij}^{z_*} \right|$, 
	\begin{align*}
	&|\hat{A}_{ij}^{\hat{z}}-\hat{A}_{ij}^{z_*}| \\
	= & \left |\frac{\sum_t G_j^{(t)} 1 (\hat{z}^{(t)}=i) }{  t_i}-\frac{\sum_t G_j^{(t)} 1 (z_*^{(t)}=i) }{  t_{i*}}  \right | \\
	\leq & \left |\frac{\sum_t G_j^{(t)} 1 (\hat{z}^{(t)}=i) }{  t_i}-\frac{\sum_t G_j^{(t)} 1 (\hat{z}^{(t)}=i) }{  t_{i*}}  \right | \\
	& +\left |\frac{\sum_t G_j^{(t)} 1 (\hat{z}^{(t)}=i) }{  t_{i*}}-\frac{\sum_t G_j^{(t)} 1 (z_*^{(t)}=i) }{  t_{i*}}  \right | \\
	\leq & \left | \frac{t_{i*}-t_{i}}{t_{i*}}  \right |+\frac{\sum_{t} \left | 1 (\hat{z}^{(t)}=i)- 1 (z_*^{(t)}=i) \right |}{t_{i*}} \\
	\leq  & M_0 n_L T_e/T,
	\end{align*}
	where $M_0$ is a constant. 
	
	Now we bound $\max_{ij} \left|\hat{A}_{ij}^{\hat{z}}-A_{ij} \right|$,
	\begin{align*}
	& \mathbb{P} \left (\max_{ij} \left|\hat{A}_{ij}^{\hat{z}}-A_{ij} \right| \geq \epsilon \right ) \\
	\leq & \mathbb{P} \left (\max_{ij} \left|\hat{A}_{ij}^{\hat{z}}-\hat{A}_{ij}^{z_*}  \right | \geq \epsilon/2 \right)+P \left (\max_{ij} \left|\hat{A}_{ij}^{z_*}-A_{ij}  \right | \geq \epsilon/2 \right) \\
	\leq & \mathbb{P} \left (M_0 n_L T_e/T \geq \epsilon \right )+P \left (\max_{ij} \left|\hat{A}_{ij}^{z_*}-A_{ij}  \right | \geq \epsilon/2 \right).
	\end{align*}
	The first term vanishes by Lemma \ref{lemma:faster}. The second term vanishes by Hoeffding's inequality: for all $\epsilon>0$,
	\begin{align*}
	&\mathbb{P} \left ( \left . \left|\hat{A}_{ij}^{z_*}-A_{ij}  \right | \geq \epsilon/2  \right |z_* \right) \\
	= &  \mathbb{P} \left (  \left . \left | \sum_t 1 (z_{*}^{(t)}=i) (G_j^{(t)}-A_{ij}) \right | \geq \epsilon t_{i*}/2 \right | z_* \right) \\
	\leq & 2\exp \{-\epsilon^2 t_{i*}/2 \}.
	\end{align*}
	Therefore, 
	\begin{align*}
	& \mathbb{P} \left (\max_{ij} \left|\hat{A}_{ij}^{z_*}-A_{ij} \right| \geq \epsilon/2 \right ) \\
	\leq & 2 n n_L  \exp \{- \epsilon^2 c_{\textnormal{min}} T/(2n_L) \} + \mathbb{P}(t_{i*} < c_\textnormal{min} T/n_L, \textnormal{for some $i$}) \rightarrow 0,
	\end{align*}
	by a similar argument in the proof of Theorem \ref{thm:separate}.
\end{proof}

\begin{proof}[Proof of Theorem \ref{thm:hub_iden_null}]
	Let $(\tilde{\rho},\tilde{A})\in \mathcal{P}$ be a set of parameters of the hub model with a null component such that $\mathbb{P}(g|\rho,A)= \mathbb{P}(g|\tilde{\rho},\tilde{A})$ for all $g$. 
	Consider the probability that no one appears:
	\begin{align*}
	\tilde{\rho}_0 \prod_{j=1}^n (1-\tilde{\pi}_j) & = \rho_0 \prod_{j=1}^n (1-\pi_j).
	\end{align*}
	For $k=n_L+1,...,n$, consider the probability that only $k$ appears:
	\begin{align*}
	\tilde{\rho}_0 \tilde{\pi}_k \prod_{j=1,...,n,j\neq k} (1-\tilde{\pi}_j) & =\rho_0 \pi_k \prod_{j=1,...,n,j\neq k} (1-\pi_j).
	\end{align*}
	From the above equations, we obtain
	\begin{align}
	& \tilde{\pi}_k  = \pi_k, \quad k=n_L+1,...,n \nonumber, \\
	& \tilde{\rho}_0 \prod_{j=1}^{n_L} (1-\tilde{\pi}_j) = \rho_0 \prod_{j=1}^{n_L} (1-\pi_j) \label{important0}.
	\end{align}
	
	For $i=1,...,n_L$, let $k$ and $k'$ be the nodes from $\{n_L+1,...,n\}$ such that $\pi_k \neq A_{ik}$ and $\pi_{k'} \neq A_{ik'}$.

	Consider the  probability that $i$ appears but no other nodes from $\{1,...,n_L\}$ appears (the rest do not matter)
	\begin{align}
	& \tilde{\rho}_0 \tilde{\pi}_i  \prod_{j=1,...,n_L, j \neq i} (1-\tilde{\pi}_j)  + \tilde{\rho}_i \prod_{j=1,...,n_L, j \neq i} (1-\tilde{A}_{ij}) \label{temp1}  \\
	= &  \rho_0 \pi_i  \prod_{j=1,...,n_L, j \neq i} (1-\pi_j)  + \rho_i \prod_{j=1,...,n_L, j \neq i} (1-A_{ij}); \nonumber
	\end{align}
	the probability that $i$ and $k$ appear but no other nodes from $\{1,...,n_L\}$ appears (the rest do not matter)
	\begin{align}
	& \tilde{\rho}_0 \tilde{\pi}_i  \prod_{j=1,...,n_L, j \neq i} (1-\tilde{\pi}_j) \pi_{k}+ \tilde{\rho}_i \prod_{j=1,...,n_L, j \neq i} (1-\tilde{A}_{ij}) \tilde{A}_{ik} \label{temp2} \\
	= &   \rho_0 \pi_i  \prod_{j=1,...,n_L, j \neq i} (1-\pi_j) \pi_{k}+ \rho_i \prod_{j=1,...,n_L, j \neq i} (1-A_{ij}) A_{ik};  \nonumber
	\end{align}
	the probability that $i$ and $k'$ appear but no other nodes from $\{1,...,n_L\}$ appears (the rest do not matter)
	\begin{align}
	& \tilde{\rho}_0 \tilde{\pi}_i  \prod_{j=1,...,n_L, j \neq i} (1-\tilde{\pi}_j) \pi_{k'}+ \tilde{\rho}_i \prod_{j=1,...,n_L, j \neq i} (1-\tilde{A}_{ij}) \tilde{A}_{ik'} \label{temp3} \\
	= &   \rho_0 \pi_i  \prod_{j=1,...,n_L, j \neq i} (1-\pi_j) \pi_{k'}+ \rho_i \prod_{j=1,...,n_L, j \neq i} (1-A_{ij}) A_{ik'}; \nonumber
	\end{align}
	and the probability that $i,k$ and $k'$ appear but no other nodes from $\{1,...,n_L\}$ appears (the rest do not matter)
	\begin{align}
	& \tilde{\rho}_0 \tilde{\pi}_i  \prod_{j=1,...,n_L, j\neq i} (1-\tilde{\pi}_j)\pi_k \pi_{k'} + \tilde{\rho}_i \prod_{l=1,...,n_L, j \neq i} (1-\tilde{A}_{ij}) \tilde{A}_{ik} \tilde{A}_{ik'} \label{temp4}  \\
	= &  \rho_0 \pi_i  \prod_{j=1,...,n_L, j\neq i} (1-\pi_j)\pi_k \pi_{k'} + \rho_i \prod_{l=1,...,n_L, j \neq i} (1-A_{ij}) A_{ik} A_{ik'}. \nonumber
	\end{align}
	Note that the above equations  are not probabilities of a single realization $g$ but are sums of multiple $\mathbb{P}(g)$. Moreover, we put $\pi_k,\pi_{k'}$ instead of $\tilde{\pi}_k,\tilde{\pi}_{k'}$ on the left-hand side of the equations, since we have proved $\tilde{\pi}_k=\pi_k,k=n_L+1,...,n$. 
	
	Let 
	\begin{align*}
	x & = \rho_0 \pi_i  \prod_{j=1,...,n_L, j \neq i} (1-\pi_j), \\
	\tilde{x} & =\tilde{\rho}_0 \tilde{\pi}_i  \prod_{j=1,...,n_L, j \neq i} (1-\tilde{\pi}_j), \\
	y & = \rho_i \prod_{j=1,...,n_L, j \neq i} (1-A_{ij}), \\
	\tilde{y} & =  \tilde{\rho}_i \prod_{l=1,...,n_L, j \neq i} (1-\tilde{A}_{ij}). 
	\end{align*}
	Then \eqref{temp1}, \eqref{temp2} \eqref{temp3} and \eqref{temp4} become
	\begin{align*}
	\tilde{x}+\tilde{y} & =x+y,  \\
	\tilde{x}\pi_k +\tilde{y} \tilde{A}_{ik} & = x\pi_k +y  A_{ik},   \\
	\tilde{x}\pi_{k'} +\tilde{y} \tilde{A}_{ik'} & = x\pi_{k'} +y  A_{ik'}, \\
	\tilde{x}\pi_k\pi_{k'}  +\tilde{y} \tilde{A}_{ik} \tilde{A}_{ik'} & = x\pi_k\pi_{k'} +y A_{ik} A_{ik'}. 
	\end{align*}
	
	Plugging $\tilde{x}-x=y-\tilde{y}$ into the last three equations, we obtain 
	\begin{align}
	\label{important1} \tilde{y} \tilde{A}_{ik} & =  \tilde{y} \pi_k+ y (A_{ik}-\pi_k), \\
	\label{important2} \tilde{y} \tilde{A}_{ik'} & =  \tilde{y} \pi_{k'}+ y (A_{ik'}-\pi_{k'}), \\
	\label{important3} y\pi_k\pi_{k'}  +\tilde{y} \tilde{A}_{ik} \tilde{A}_{ik'} & = \tilde{y} \pi_k\pi_{k'} +y A_{ik}  A_{ik'}.
	\end{align}
	Multiplying \eqref{important3} by $\tilde{y}$, and plugging the right hand sides of \eqref{important1} and \eqref{important2} into the resulting equation, we obtain
	\begin{align*}
	& y\tilde{y}\pi_k\pi_{k'}  + \tilde{y}^2 \pi_k \pi_{k'}+\tilde{y} \pi_k y (A_{ik'}-\pi_{k'})+ \tilde{y} \pi_{k'} y (A_{ik}-\pi_k)+y^2 (A_{ik}-\pi_k)(A_{ik'}-\pi_{k'}) \\
	& = \tilde{y}^2 \pi_k\pi_{k'} +y\tilde{y} A_{ik}  A_{ik'} \\
	\Rightarrow &  y (A_{ik}-\pi_k)(A_{ik'}-\pi_{k'}) = \tilde{y} (A_{ik}-\pi_k)(A_{ik'}-\pi_{k'}).
	\end{align*}
	Therefore,  
	$
	\tilde{y}=y
	$ since $\pi_k \neq A_{ik}$ and $\pi_{k'} \neq A_{ik'}$. It follows that $\tilde{x}=x$, i.e., 
	\begin{align*}
	\tilde{\rho}_0 \tilde{\pi}_i  \prod_{j=1,...,n_L, j \neq i} (1-\tilde{\pi}_j) = \rho_0 \pi_i  \prod_{j=1,...,n_L, j \neq i} (1-\pi_j), \quad i=1,...,n_L.
	\end{align*}
	
	Combining the above equation with \eqref{important0}, we obtain 
	\begin{align*}
	&\tilde{\pi}_i = \pi_i, \quad i=1,...,n_L, \\
	& \tilde{\rho}_0=\rho_0.
	\end{align*}
	
	Note that $\mathbb{P}(g)=\mathbb{P}(g|z=0)\mathbb{P}(z=0)+\mathbb{P}(g|z\neq 0)\mathbb{P}(z\neq 0)$. So far we have proved parameters of $\mathbb{P}(g|z=0)$ and $\mathbb{P}(z=0)$ are identifiable. We only need to prove the identifiability of $\mathbb{P}(g|z\neq 0)$, which is the case of the asymmetric hub model and has been proved by Theorem \ref{thm:hub_iden}.
\end{proof}

\begin{proof}[Proof of Theorem \ref{thm:ti0}]
	First, we can use the same argument in Theorem \ref{thm:separate} to prove that there exists $\delta_1>0$ such that  
	\begin{align}
	& t_{ii}+t_{i0} \geq \delta_1 t_{i*},   \quad i =1,...,n_L, \label{tricky}\\
	& t_{00} \geq \delta_1 t_{0*},  \label{t00} 
	\end{align}
	with probability approaching 1.
	
	Therefore\footnote{Some inequalities below hold with probability approaching 1. We omit this sentence occasionally.}, for $i=1,...,n_L$, $j\in V_i$,
	\begin{align*}
	A_{ij}-\bar{A}_{0j} & = \frac{\sum_t (A_{ij}-P_j^{(t)}) 1(\hat{z}^{(t)}=0)}{t_0} \\
	& = \frac{\sum_{k=0}^{n_L} (A_{ij}-A_{kj}) t_{k0}}{t_0} \\
	& \geq \frac{(A_{ij}-A_{0j})t_{00}}{t_0}	\geq d \frac{t_{00}}{T} \geq d \frac{t_{00}}{(n_L+1)t_{0*}/c_{\textnormal{min}}} \geq  \frac{dc_{\textnormal{min}} \delta_1}{n_L+1}.
	\end{align*}
	It follows that
	\begin{align}
	L_P(z_*)-L_P(\hat{z})  \geq & \max_{i=1,...,n_L} \sum_{t: z_*^{(t)}=i, \hat{z}^{(t)}=0}  \sum_{j \in V_i} 2(A_{ij}-\bar{A}_{0j})^2 \nonumber \\
	\geq & \max_{i=1,...,n_L} 2 \left ( \frac{dc_{\textnormal{min}} \delta_1}{n_L+1} \right)^2 \frac{vn}{n_L+1} t_{i0} \nonumber \\
	\geq & \max_{i=1,...,n_L} 2 \left ( \frac{dc_{\textnormal{min}} \delta_1}{n_L+1} \right)^2  \frac{vn}{n_L+1} \frac{t_{i0}}{t_{i*}} \frac{c_{\textnormal{min}}T}{n_L+1} \nonumber \\
	\geq & \max_{i=1,...,n_L} \frac{ d^2 v nT}{\delta n_L^4} \frac{t_{i0}}{t_{i*}}, \label{thm:separate_temp}
	\end{align}
	where $ \delta$ is a positive constant. 
	
	Also note that by using the same argument in Theorem \ref{thm:label_consistency},
	\begin{align*}
	\max_z \frac{  n_L^4 }{d^2 v nT}|L_G(z)-L_P(z)| =o_p(1)
	\end{align*}
	holds under the condition $ (n_L^5\log T)/(d^2 v T)=o(1) $ and $ (n_L^8 \log n_L )/ (d^4 v^2 n)=o(1) $. Therefore, 
	\begin{align}
	\mathbb{P}\left ( \max_{i=1,...,n_L}\frac{t_{i0}}{t_{i*}} \geq \eta \right ) \rightarrow 0. \label{tricky_solved}
	\end{align}
\end{proof}

\begin{proof}[Proof of Theorem \ref{thm:label_consistency_null}]
	Due to \eqref{tricky} and \eqref{tricky_solved}, there exists $\delta_2>0$ such that 
	\begin{align*}
	t_{ii}\geq \delta_2 t_{i*}   \quad \mbox{for } i =0,...,n_L,
	\end{align*}
	with probability approaching 1.
	
	For $i=1,...,n_L$, $k=0,...,n_L$, $k \neq i$ and $j\in V_i$,
	\begin{align*}
	\frac{t_{ik}}{t_k} \leq \frac{t_{ik}}{t_{ik}+t_{kk}} =\frac{t_{ik}/t_{i*}}{t_{ik}/t_{i*} +t_{kk}/t_{k*} \cdot t_{k*}/t_{i*}} \leq \frac{1}{1+\delta_2\cdot c_{\textnormal{max}}/c_{\textnormal{min}}}  = \delta_3,
	\end{align*}
	\begin{align*}
	A_{ij}-\bar{A}_{kj} & = \frac{\sum_t (A_{ij}-P_j^{(t)}) 1(\hat{z}^{(t)}=k)}{t_k} \\
	& = \frac{\sum_{k'=0}^{n_L} (A_{ij}-A_{k'j}) t_{k'k}}{t_k} \\
	& \geq \frac{d \sum_{k' \neq i} t_{k'k}}{t_k}  \geq d (1-\delta_3).
	\end{align*}
	
	Now note that with probability approaching 1, 
	\begin{align*}
	L_P(z_*)-L_P(\hat{z})  \geq & \sum_{t=1}^T \sum_{j=1}^n 2(P_j^{(t)}-\bar{A}_{\hat{z}^{(t)},j})^2 \\
	\geq & \sum_{i=1}^{n_L} \sum_{0\leq k \leq n_L, k\neq i} \sum_{t: z_*^{(t)}=i, \hat{z}^{(t)}=k}  \sum_{j \in V_i} 2(A_{ij}-\bar{A}_{kj})^2 \\
	\geq &   \frac{2vn}{n_L}  \sum_{i=1}^{n_L} \sum_{0\leq k \leq n_L, k\neq i} t_{ik} d^2 (1-\delta_3)^2,
	\end{align*}
	which implies that there exists $\delta>0$ such that with probability approaching 1,
	\begin{align}\label{sep11}
	\frac{\delta n_L}{d^2 v n T } (L_P(z_*)-L_P(\hat{z})) \geq  \sum_{i=1}^{n_L} \sum_{0\leq k \leq n_L, k\neq i} \frac{ t_{ik}}{T}.
	\end{align}
	By the same argument in Theorem \ref{thm:label_consistency}, this further implies 
	\begin{align}
	\sum_{i=1}^{n_L} \sum_{0\leq k \leq n_L, k\neq i} \frac{ t_{ik}}{T} =o_p(1), \quad \mbox{as } n_L\rightarrow \infty, n\rightarrow \infty, T\rightarrow \infty, \label{part1}
	\end{align}
	if $ (n_L^2\log T)/(d^2 v T)=o(1) $ and $ (n_L^2 \log n_L )/ (d^4 v^2 n)=o(1)$. 
	
	Similarly,
	\begin{align}
	\sum_{i=1}^{n_L} \sum_{0\leq k \leq n_L, k\neq i} \frac{ (n_L+1) t_{ik}}{T} =o_p(1), \quad \mbox{as } n_L\rightarrow \infty, n\rightarrow \infty, T\rightarrow \infty, \label{novel} 
	\end{align}
	if $ (n_L^3\log T)/(d^2 v T)=o(1) $ and $ (n_L^4 \log n_L )/ (d^4 v^2 n)=o(1) $. 
	
	Now we bound $t_{0i}$, $i=1,...,n_L$.     From \eqref{novel},
	$\sum_{1\leq k\leq n_L,k\neq i} t_{ki} = o_p( T/(n_L+1))$. And from $\delta_2 Tc_{\textnormal{min}}/(n_L+1) \leq  \delta_2 t_{i*} \leq t_{ii}$, $\sum_{1\leq k\leq n_L,k\neq i} t_{ki} \leq t_{ii}$, with probability approaching 1. Moreover, from \eqref{t00}, $t_{0i}\leq (1-\delta_1) t_{0*}$.
	
	Therefore,  there exists $\delta_4>0$ such that for $i=1,...,n_L$, $j\in V_i$,
	\begin{align*}
	A_{ij}-\bar{A}_{ij} & = \frac{\sum_t (A_{ij}-P_j^{(t)}) 1(\hat{z}^{(t)}=i)}{t_i} \\
	& \geq \frac{(A_{ij}-A_{0j})t_{0i}}{t_i} \\
	& \geq \frac{dt_{0i}}{t_{0i}+t_{ii}+\sum_{1\leq k\leq n_L,k\neq i} t_{ki}} \\
	& \geq \frac{dt_{0i}}{(1-\delta_1)t_{0*}+2 t_{ii}} \\
	& \geq \frac{dt_{0i}}{(1-\delta_1)t_{0*}+2 t_{i*}} \geq \frac{d n_L t_{0i}}{\delta_4 T}.
	\end{align*}
	It follows that
	\begin{align}
	L_P(z_*)-L_P(\hat{z})  \geq & \max_{i=1,...,n_L} \sum_{t: z_*^{(t)}=i, \hat{z}^{(t)}=i}  \sum_{j \in V_i} 2(A_{ij}-\bar{A}_{ij})^2  \nonumber \\
	\geq & \max_{i=1,...,n_L} 2 \left ( \frac{d n_L t_{0i}}{\delta_4 T} \right)^2 \frac{vn}{n_L+1} t_{ii} \nonumber  \\
	\geq & \max_{i=1,...,n_L} 2 \left ( \frac{d}{\delta_4} \right)^2 \left ( \frac{n_L t_{0i}}{T} \right )^2   \frac{vn}{n_L+1}  \delta_2 t_{i*} \nonumber \\
	\geq & \max_{i=1,...,n_L} 2 \left ( \frac{d}{\delta_4} \right)^2 \left ( \frac{n_L t_{0i}}{T} \right )^2    \frac{vn}{n_L+1}  \delta_2 T \frac{c_\textnormal{min}}{n_L+1} \nonumber  \\
	\geq & \max_{i=1,...,n_L} \frac{ d^2 v nT}{\delta' n_L^2} \left ( \frac{n_L t_{0i}}{T} \right )^2,  \label{sep22}
	\end{align}
	where $\delta'$ is positive constant. 
	
	By using the same argument in Theorem \ref{thm:label_consistency}, 
	\begin{align}
	\max_{i=1,...,n_L}  \frac{n_L t_{0i}}{T} =o_p(1),  \label{part2}
	\end{align}
	if $ (n_L^3\log T)/(d^2 v T)=o(1) $ and $ (n_L^4 \log n_L )/ (d^4 v^2 n)=o(1) $. 
	
	It follows that
	\begin{align*}
	\sum_{i=1}^{n_L} \frac{t_{0i}}{T}=o_p(1), 
	\end{align*}
	Combining \eqref{part1} and \eqref{part2},
	\begin{align*}
	\frac{T_e}{T} =o_p(1), \quad \mbox{as } n_L\rightarrow \infty, n\rightarrow \infty, T\rightarrow \infty.
	\end{align*}	
\end{proof}

\begin{proof}[Proof of Theorem \ref{thm:estimation_null}]
	By the same argument in Theorem \ref{thm:estimation}, it is sufficient to show  
	\begin{align}\label{final}
	\frac{(n_L+1) T_e}{T} =o_p(1), \quad \mbox{as } n_L\rightarrow \infty, n\rightarrow \infty, T\rightarrow \infty.
	\end{align}	
	We have proved 
	\begin{align}
	\sum_{i=1}^{n_L} \sum_{0\leq k \leq n_L, k\neq i} \frac{ (n_L+1) t_{ik}}{T} =o_p(1), \quad \mbox{as } n_L\rightarrow \infty, n\rightarrow \infty, T\rightarrow \infty,
	\end{align}
	if $ (n_L^3\log T)/(d^2 v T)=o(1) $ and $ (n_L^4 \log n_L )/ (d^4 v^2 n)=o(1) $. 
	
	From \eqref{sep22}, there exists $\delta'>0$ such that
	\begin{align*}
	L_P(z_*)-L_P(\hat{z})  \geq    \max_{i=1,...,n_L} \frac{ d^2 v nT}{\delta' n_L^4} \left ( \frac{n_L(n_L+1) t_{0i}}{T} \right )^2,
	\end{align*}
	which further implies
	\begin{align*}
	\max_{i=1,...,n_L}  \frac{n_L(n_L+1)t_{0i}}{T} =o_p(1), 
	\end{align*}
	if $ (n_L^5\log T)/(d^2 v T)=o(1) $ and $ (n_L^8 \log n_L )/ (d^4 v^2 n)=o(1) $.
	
	It follows that
	\begin{align*}
	\sum_{i=1}^{n_L} \frac{(n_L+1)t_{0i}}{T}=o_p(1), 
	\end{align*}
	and \eqref{final} is therefore proved and so is the theorem. 
\end{proof}

 \newcommand{\noop}[1]{}

\end{document}